\documentclass[12pt]{amsart}
\usepackage{graphicx}
\usepackage{amsmath}
\usepackage{amssymb}
\usepackage{setspace}
\newtheorem{thm}{Theorem}[section]
\newtheorem{cor}[thm]{Corollary}

\newtheorem{lem}[thm]{Lemma}

\theoremstyle{definition}
\newtheorem{defn}[thm]{Definition}
\theoremstyle{remark}
\newtheorem{rem}[thm]{Remark}
\theoremstyle{conclusion}

\theoremstyle{question}
\numberwithin{equation}{section}

\begin{document}
\title[Critical order H\'{e}non-Lane-Emden type equations on $\mathbb{R}^{n}_{+}$]{Liouville type theorem for critical order H\'{e}non-Lane-Emden type equations on a half space and its applications}

\author{Wei Dai, Guolin Qin}

\address{School of Mathematics and Systems Science, Beihang University (BUAA), Beijing 100083, P. R. China, and LAGA, Universit\'{e} Paris 13 (UMR 7539), Paris, France}
\email{weidai@buaa.edu.cn}

\address{Institute of Applied Mathematics, Chinese Academy of Science, Beijing 100190, and University of Chinese Academy of science, Beijing 100049, P. R. China}
\email{qinguolin18@mails.ucas.ac.cn}

\thanks{Wei Dai is supported by the NNSF of China (No. 11501021) and the State Scholarship Fund of China (No. 201806025011).}

\begin{abstract}
In this paper, we are concerned with the critical order H\'{e}non-Lane-Emden type equations with Navier boundary condition on a half space $\mathbb{R}^n_+$:
\begin{equation}\label{NPDE0}\\\begin{cases}
(-\Delta)^{\frac{n}{2}} u(x)=f(x,u(x)), \,\,\,\,\,\, \,\, u(x)\geq0, \,\,\,\, \,\,\,\,\, x\in\mathbb{R}^{n}_+, \\
u=(-\Delta)u=\cdots=(-\Delta)^{\frac{n}{2}-1}u=0, \,\,\,\,\,\,\,\,\,\, \,\,\,  x\in\partial\mathbb{R}^{n}_+,
\end{cases}\end{equation}
where $u\in C^{n}(\mathbb{R}^{n}_+)\cap C^{n-2}(\overline{\mathbb{R}^{n}_+})$ and $n\geq2$ is even. We first consider the typical case $f(x,u)=|x|^{a}u^{p}$ with $0\leq a<\infty$ and $1<p<\infty$. We prove the super poly-harmonic properties and establish the equivalence between \eqref{NPDE0} and the corresponding integral equations
\begin{equation}\label{IE0}
	u(x)=\int_{\mathbb{R}^{n}_+}G(x,y)f(y,u(y))dy,
\end{equation}
where $G(x,y)$ denotes the Green's function for $(-\Delta)^{\frac{n}{2}}$ on $\mathbb{R}^{n}_{+}$ with Navier boundary conditions. Then, we establish Liouville theorem for \eqref{IE0} via ``the method of scaling spheres" developed initially in \cite{DQ0} by Dai and Qin, and hence we obtain the Liouville theorem for \eqref{NPDE0} on $\mathbb{R}^n_+$. As an application of the Liouville theorem on $\mathbb{R}^n_+$ (Theorem \ref{thm1}) and Liouville theorems in $\mathbb{R}^{n}$ established in Chen, Dai and Qin \cite{CDQ} for $n\geq4$ and Bidaut-V\'{e}ron and Giacomini \cite{BG} for $n=2$, we derive a priori estimates and existence of positive solutions to critical order Lane-Emden equations in bounded domains for all $n\geq2$ and $1<p<\infty$. Extensions to IEs and PDEs with general nonlinearities $f(x,u)$ are also included.
\end{abstract}
\maketitle {\small {\bf Keywords:} The method of scaling spheres; Critical order; H\'{e}non-Lane-Emden type equations; Liouville theorems; a priori estimates; Navier problems.\\

}

\section{Introduction}

\subsection{Liouville theorems on a half space $\mathbb{R}^{n}_{+}$}

In this paper, we first establish Liouville theorem for the following higher order H\'{e}non-Lane-Emden equations with Navier boundary condition:
\begin{equation}\label{NPDE}\\\begin{cases}
(-\Delta)^{\frac{n}{2}} u(x)=|x|^{a}u^{p}(x), \,\,\,\,\,\,\, u(x)\geq0, \,\,\,\,\,\,\,\, \,\,\, x\in\mathbb{R}^{n}_+, \\
u=(-\Delta)u=\cdots=(-\Delta)^{\frac{n}{2}-1}u=0 \,\,\,\,\,\,\,\,\,\,\, \text{on} \,\,\,  \partial\mathbb{R}^{n}_+,
\end{cases}\end{equation}
where $\mathbb{R}^{n}_{+}=\{x=(x_{1},\cdots,x_{n})\in\mathbb{R}^{n}\,|\,x_{n}>0\}$ is the upper half Euclidean space, $u\in C^{n}(\mathbb{R}^{n}_+)\cap C^{n-2}(\overline{\mathbb{R}^{n}_+})$, $n\geq2$ is even, $0\leq a<+\infty$ and $1<p<+\infty$.

For $0<\alpha\leq n$, PDEs of the form
\begin{equation}\label{GPDE}
(-\Delta)^{\frac{\alpha}{2}}u(x)=|x|^{a}u^{p}(x)
\end{equation}
are called the fractional order or higher order H\'{e}non, Lane-Emden, Hardy equations for $a>0$, $a=0$, $a<0$, respectively. These equations have numerous important applications in conformal geometry and Sobolev inequalities. In particular, in the case $a=0$, \eqref{GPDE} becomes the well-known Lane-Emden equation, which models many phenomena in mathematical physics and astrophysics. When $a>0$, equations of type \eqref{GPDE} was first proposed by H\'{e}non in \cite{Henon} when he studied rotating stellar structures.

We say equations \eqref{GPDE} have critical order if $\alpha=n$ and non-critical order if $0<\alpha<n$. Being essentially different from the non-critical order equations, the fundamental solution $c_{n}\ln\frac{1}{|x-y|}$ of $(-\Delta)^{\frac{n}{2}}$ changes its signs in critical order case $\alpha=n$. The nonlinear terms in \eqref{GPDE} are called critical if $p=p_{s}(a):=\frac{n+\alpha+2a}{n-\alpha}$ ($:=+\infty$ if $n=\alpha$) and subcritical if $0<p<p_{s}(a)$.

Liouville type theorems (i.e., nonexistence of nontrivial nonnegative solutions) and related properties for equations \eqref{GPDE} in the whole space $\mathbb{R}^n$ and the half space $\mathbb{R}^n_+$ have been extensively studied (see \cite{CDQ,CFL,CFY,CL,CLZ,Cowan,DPQ,DQ0,DQ1,DQ,DQZ,DZ,GS,Lin,Li,LB,M,MP,PS,RW,RW1,SX,WX} and the references therein). These Liouville theorems, in conjunction with the blowing up and re-scaling arguments, are crucial in establishing a priori estimates and hence existence of positive solutions to non-variational boundary value problems for a class of elliptic equations on bounded domains or on Riemannian manifolds with boundaries (see \cite{CDQ,CL4,DQ0,GS1,PQS}).

In the critical order case $\alpha=n$, Liouville theorems for \eqref{GPDE} have been established in \cite{BG} for $n=2$ and in \cite{CDQ} for $n\geq 4$ in whole space $\mathbb{R}^{n}$. For the special case $a=0$, their results can be concluded as the following theorem.
\begin{thm}\label{wthm}(\cite{BG,CDQ})
	Suppose $n\geq 2$ is an even integer, $1<p<+\infty$ and $u$ is a nonnegative classical solution of
	\begin{equation*}
	(-\Delta)^{\frac{n}{2}}u(x)=u^{p}(x), \qquad \forall \,\, x\in\mathbb{R}^{n}.
	\end{equation*}
	Then, we have $u\equiv 0$ in $\mathbb{R}^{n}$.
\end{thm}

For $a=0$, there are also many works on the Liouville type theorems for Lane-Emden equations on half space $\mathbb{R}^{n}_{+}$, for instance, see \cite{CFL,CFY,CLZ,CLZC,DQ0,DQZ,LZ,RW,RW1} and the references therein. Reichel and Weth \cite{RW} proved Liouville theorem in the class of \emph{bounded} nonnegative solutions for Dirichlet problem of higher order Lane-Emden equations \eqref{GPDE} (i.e., $a=0$ and $\alpha=2m$ with $1\leq m<\frac{n}{2}$) on $\mathbb{R}^{n}_{+}$ in the cases $1<p\leq\frac{n+2m}{n-2m}$, subsequently they also derived in \cite{RW1} Liouville theorem for general nonnegative solutions in the cases $1<p<\frac{n+2m}{n-2m}$. In \cite{CFL}, Chen, Fang and Li established Liouville theorem for Navier problem of Lane-Emden equation \eqref{GPDE} on $\mathbb{R}^{n}_{+}$ in the higher order cases $\alpha=2m$ with $1\leq m<\frac{n}{2}$ and $\frac{n}{n-2m}<p\leq\frac{n+2m}{n-2m}$. In a recent work \cite{DQ0}, Dai and Qin developed the method of scaling spheres, which is essentially a frozen variant of the method of moving spheres initially used by Chen and Li \cite{CL1}, Li and Zhu \cite{LZ} and Padilla \cite{P} and becomes a powerful tool in deriving asymptotic estimates for solutions. As one of many immediate applications, they established in \cite{DQ0} the Liouville theorem for non-critical higher order H\'{e}non-Lane-Emden type IEs and Lane-Emden type PDEs with Navier boundary conditions on $\mathbb{R}^{n}_{+}$ for all $1<p\leq\frac{n+2m}{n-2m}$.

For Liouville theorem on $\mathbb{R}^{n}_{+}$, the cases $a\neq0$ have not been fully understood. For $a\geq0$, by using the method of scaling spheres developed initially in \cite{DQ0}, Dai, Qin and Zhang \cite{DQZ} proved the Liouville theorem for non-critical higher order H\'{e}non equations \eqref{GPDE} (i.e., $\alpha=2m$ with $1\leq m<\frac{n}{2}$) with Navier boundary conditions on $\mathbb{R}^{n}_{+}$ in the cases $1<p<\frac{n+2m+2a}{n-2m}$.

In this paper, by applying the method of scaling spheres in integral forms, we will establish Liouville theorem for the Navier problem of critical order H\'{e}non-Lane-Emden equation \eqref{NPDE} on $\mathbb{R}^{n}_{+}$ in all the cases that $a\geq 0$, $n\geq2$ and $1<p<+\infty$.

It's well known that the super poly-harmonic properties of solutions are crucial in establishing Liouville type theorems and the integral representation formulae for higher order or fractional order PDEs (see e.g. \cite{CDQ,CF,CFL,DPQ,DQZ,WX}). In order to prove the equivalence between PDE \eqref{NPDE} and corresponding integral equation, we will first prove the following generalized theorem on super poly-harmonic properties, namely, we allow $-n<a<0$ and assume that $u\in C^{n}(\mathbb{R}^{n}_+)\cap C^{n-2}(\overline{\mathbb{R}^{n}_+})$ if $-n<a<0$.
\begin{thm}\label{Thm0}(Super poly-harmonic properties)
Assume $n\geq4$ is even, $-n<a<+\infty$, $1<p<+\infty$ and $u$ is a nonnegative solution of \eqref{NPDE}. If one of the following two assumptions
\begin{equation*}
	a\geq-2p-2 \,\,\,\,\,\,\,\,\,\,\,\, \text{or} \,\,\,\,\,\,\,\,\,\,\,\, u(x)=o(|x|^{2}) \,\,\,\, \text{as} \,\, |x|\rightarrow+\infty
\end{equation*}
holds, then
\begin{equation*}
  (-\Delta)^{i}u(x)\geq0
\end{equation*}
for every $i=1,2,\cdots,\frac{n}{2}-1$ and all $x\in\overline{\mathbb{R}^{n}_+}$.
\end{thm}

Based on the above super poly-harmonic properties, we can deduce the equivalence between PDE \eqref{NPDE} and the following integral equation
\begin{equation}\label{IE}
	u(x)=\int_{\mathbb{R}^{n}_{+}}G^{+}(x,y)|y|^{a} u^{p}(y)dy,
\end{equation}
where
\begin{equation}\label{Green}
  G^{+}(x,y):=C_{n}\left(\ln{\frac{1}{|x-y|}}-\ln{\frac{1}{|\bar{x}-y|}}\right)
\end{equation}
denotes the Green's function for $(-\Delta)^{\frac{n}{2}}$ on $\mathbb{R}^{n}_{+}$ with Navier boundary conditions, and $\bar{x}:=(x_{1},\cdots,-x_{n})$ is the reflection of $x$ with respect to the boundary $\partial\mathbb{R}^{n}_{+}$. That is, we have the following theorem.
\begin{thm}\label{equivalence}
If u is a nonnegative classical solution of \eqref{NPDE}, then $u$ is also a nonnegative solution of integral equation \eqref{IE}, and vice versa.
\end{thm}

Next, we consider the integral equations \eqref{IE} instead of PDE \eqref{NPDE}. We will study the integral equation \eqref{IE} via the method of scaling spheres in integral forms developed by Dai and Qin in \cite{DQ0}. Our Liouville type result for IE \eqref{IE} is the following theorem.
\begin{thm}\label{thmIE}
Assume $n\geq 1$, $1\leq p<+\infty$ and $-n<a<+\infty$. If $u\in C(\overline{\mathbb{R}^{n}_{+}})$ is a nonnegative solution to \eqref{IE}, then $u\equiv 0$.
\end{thm}

\begin{rem}
Note that we do not assume $n$ to be even in Theorem \ref{thmIE}. It is also unexpected that the above Theorem \ref{thmIE} still holds for $n=1$. One can see clearly from the proof that the assumption $u\in C(\overline{\mathbb{R}^{n}_{+}})$ in Theorem \ref{thmIE} can be weaken into $|x|^{a}u^{p-1}\in L^{1+\delta}_{loc}(\overline{\mathbb{R}^{n}_{+}})$ for some small $\delta>0$.
\end{rem}

As a consequence of Theorem \ref{equivalence} and \ref{thmIE}, we obtain immediately the following Liouville type theorem on PDE \eqref{NPDE}.
\begin{thm}\label{thm1}
Assume $n\geq2$ is even, $0\leq a<+\infty$ and $1<p<+\infty$. Suppose $u\in C^{n}({\mathbb{R}^{n}_+})\cap C^{n-2}(\overline{\mathbb{R}^{n}_+})$ is a nonnegative classical solution to \eqref{NPDE}, then $u\equiv 0$.
\end{thm}

\subsection{A priori estimates and existence of positive solutions in bounded domains}

As an immediate application of the Liouville theorems (Theorem \ref{wthm} for $\mathbb{R}^{n}$ and Theorem \ref{thm1} for $\mathbb{R}^{n}_{+}$), we can derive a priori estimates and existence of positive solutions to critical order Lane-Emden equations in bounded domains $\Omega$ for all $1<p<+\infty$.

In general, let the critical order uniformly elliptic operator $L$ be defined by
\begin{eqnarray}\label{0-0c}
  L &:=& \left(\sum_{i,j=1}^{n}a_{ij}(x)\frac{\partial^{2}}{\partial x_{i}\partial x_{j}}\right)^{\frac{n}{2}}+\sum_{|\beta|\leq n-1}b_{\beta}(x)D^{\beta} \\
 \nonumber &=:& A^{\frac{n}{2}}+\sum_{|\beta|\leq n-1}b_{\beta}(x)D^{\beta},
\end{eqnarray}
where $n\geq2$ is even and the coefficients $b_{\beta}\in L^{\infty}(\Omega)$ and $a_{ij}\in C^{n-2}(\overline{\Omega})$ such that there exists constant $\tau>0$ with
\begin{equation}\label{0-2c}
  \tau|\xi|^{2}\leq\sum_{i,j=1}^{n}a_{ij}(x)\xi_{i}\xi_{j}\leq\tau^{-1}|\xi|^{2}, \,\,\,\,\,\,\,\,\,\, \forall \,\, \xi\in\mathbb{R}^{n}, \,\, x\in\Omega.
\end{equation}
Consider the Navier boundary value problem:
\begin{equation}\label{PDE-N}\\\begin{cases}
Lu(x)=f(x,u), \,\,\,\,\,\,\,\,\,\,\,\,\,\,\,\, x\in\Omega, \\
u(x)=Au(x)=\cdots=A^{\frac{n}{2}-1}u(x)=0, \,\,\,\,\,\,\,\, x\in\partial\Omega,
\end{cases}\end{equation}
where $n\geq2$ is even, $u\in C^{n}(\Omega)\cap C^{n-2}(\overline{\Omega})$ and $\Omega$ is a bounded domain with boundary $\partial\Omega\in C^{n-2}$.

By virtue of the Liouville theorem in $\mathbb{R}^{n}$ established in \cite{BG,CDQ} (see also Theorem \ref{wthm}) and Liouville theorem in $\mathbb{R}^{n}_{+}$ (Theorem \ref{thm1}), using entirely similar blowing-up and re-scaling methods as in the proof of Theorem 6 in Chen, Fang and Li \cite{CFL}, we can derive the following a priori estimate for classical solutions (possibly sign-changing solutions) to the critical order Navier problem \eqref{PDE-N} in the full range $1<p<+\infty$.
\begin{thm}\label{Thm4c}
Assume $n\geq2$ is even, $1<p<+\infty$ and there exist positive, continuous functions $h(x)$ and $k(x)$: $\overline{\Omega}\rightarrow(0,+\infty)$ such that
\begin{equation}\label{0-3c}
  \lim_{s\rightarrow+\infty}\frac{f(x,s)}{s^{p}}=h(x), \,\,\,\,\,\,\,\,\,\,\,\, \lim_{s\rightarrow-\infty}\frac{f(x,s)}{|s|^{p}}=k(x)
\end{equation}
uniformly with respect to $x\in\overline{\Omega}$. Then there exists a constant $C>0$ depending only on $\Omega$, $n$, $p$, $h(x)$, $k(x)$, such that
\begin{equation}\label{0-4c}
  \|u\|_{L^{\infty}(\overline{\Omega})}\leq C
\end{equation}
for every classical solution $u$ of Navier problem \eqref{PDE-N}.
\end{thm}

\begin{rem}\label{remark3c}
The proof of Theorem \ref{Thm4c} is entirely similar to that of Theorem 6 in \cite{CFL} (see also Theorem 1.13 in \cite{DQ0}). We only need to replace the Liouville theorems for non-critical order Lane-Emden equations in $\mathbb{R}^{n}$ (see Lin \cite{Lin} for fourth order and Wei and Xu \cite{WX} for general even order) by Liouville theorems for critical order equations in $\mathbb{R}^{n}$ (see Bidaut-V\'{e}ron and Giacomini \cite{BG} for $n=2$ and Chen, Dai and Qin \cite{CDQ} for $n\geq4$, see also Theorem \ref{wthm}), and replace the Liouville theorems for non-critical order Lane-Emden equations on $\mathbb{R}^{n}_{+}$ (Theorem 5 in \cite{CFL}, or further, Theorem 1.10 in \cite{DQ0}) by Theorem \ref{thm1} in the proof. Thus we omit the details of the proof.
\end{rem}

One can immediately apply Theorem \ref{Thm4c} to the following critical order Navier problem:
\begin{equation}\label{tNavier}\\\begin{cases}
(-\Delta)^{\frac{n}{2}}u(x)=u^{p}(x)+t \,\,\,\,\,\,\,\,\,\, \text{in} \,\,\, \Omega, \\
u(x)=-\Delta u(x)=\cdots=(-\Delta)^{\frac{n}{2}-1}u(x)=0 \,\,\,\,\,\,\,\, \text{on} \,\,\, \partial\Omega,
\end{cases}\end{equation}
where $n\geq2$, $\Omega\subset\mathbb{R}^{n}$ is a bounded domain with boundary $\partial\Omega\in C^{n-2}$ and $t$ is an arbitrary nonnegative real number.

We can deduce the following corollary from Theorem \ref{Thm4c}.
\begin{cor}\label{cor1c}
Assume $1<p<+\infty$. Then, for any nonnegative solution $u\in C^{n}(\Omega)\cap C^{n-2}(\overline{\Omega})$ to the critical order Navier problem \eqref{tNavier}, we have
	\begin{equation}\label{0-5c}
	\|u\|_{L^{\infty}(\overline{\Omega})}\leq C(n,p,\Omega).
	\end{equation}
\end{cor}

\begin{rem}\label{remark5c}
In \cite{CDQ}, due to the lack of Liouville theorem in $\mathbb{R}^{n}_{+}$ (Theorem \ref{thm1}), the authors first applied the method of moving planes in local way to derive a boundary layer estimates, then by using blowing-up arguments (see \cite{BM,CL}), they could only establish the a priori estimates for the critical order Navier problem \eqref{tNavier} under the assumptions that either $1<p<+\infty$ and $\Omega$ is strictly convex, or $1<p\leq \frac{n+2}{n-2}$ (see Theorem 1.3 in \cite{CDQ}). Now, as an immediate consequence of Theorem \ref{Thm4c}, we derive in Corollary \ref{cor1c} a priori estimates for the critical order Navier problem \eqref{tNavier} for all the cases $1<p<+\infty$ with no convexity assumptions on $\Omega$, which extends Theorem 1.3 in \cite{CDQ} remarkably.
\end{rem}

As a consequence of the a priori estimates (Corollary \ref{cor1c}), by applying the Leray-Schauder fixed point theorem (see Theorem 4.1 in \cite{CDQ}), we can derive existence result for positive solution to the following Navier problem for critical order Lane-Emden equations in the full range $1<p<+\infty$:
\begin{equation}\label{Navier}\\\begin{cases}
(-\Delta)^{\frac{n}{2}}u(x)=u^{p}(x) \,\,\,\,\,\,\,\,\,\, \text{in} \,\,\, \Omega, \\
u(x)=-\Delta u(x)=\cdots=(-\Delta)^{\frac{n}{2}-1}u(x)=0 \,\,\,\,\,\,\,\, \text{on} \,\,\, \partial\Omega,
\end{cases}\end{equation}
where $n\geq 2$ is even, $1<p<+\infty$ and $\Omega\subset\mathbb{R}^{n}$ is a bounded domain with boundary $\partial\Omega\in C^{n-2}$.

By virtue of the a priori estimates (Theorem 1.3 in \cite{CDQ}), using the Leray-Schauder fixed point theorem, Chen, Dai and Qin \cite{CDQ} obtained existence of positive solution for the critical order Navier problem \eqref{Navier} under the assumptions that either $p\in(1,+\infty)$ and $\Omega$ is strictly convex, or $p\in\left(1,\frac{n+2}{n-2}\right]$ (Theorem 1.4 in \cite{CDQ}). For existence results on non-critical higher order H\'{e}non-Hardy equations on bounded domains, please see \cite{CPY,CLP,DPQ,DQ0,GGN,N} and the references therein. Since Corollary \ref{cor1c} extends Theorem 1.3 in \cite{CDQ} to the full range $1<p<+\infty$ with no convexity assumptions on $\Omega$, through entirely similar arguments, we can improve Theorem 1.4 in \cite{CDQ} remarkably and derive the following existence result for positive solution to the critical order Navier problem \eqref{Navier} in the full range $1<p<+\infty$.

\begin{thm}\label{Thm5}
Assume $1<p<+\infty$. Then, the critical order Navier problem \eqref{Navier} possesses at least one positive solution $u\in C^{n}(\Omega)\cap C^{n-2}(\overline{\Omega})$. Moreover, the positive solution $u$ satisfies
	\begin{equation}\label{lower-bound}
	\|u\|_{L^{\infty}(\overline{\Omega})}\geq\left(\frac{\sqrt{2n}}{diam\,\Omega}\right)^{\frac{n}{p-1}}.
	\end{equation}
\end{thm}

\begin{rem}\label{remark4c}
The proof of Theorem \ref{Thm5} is entirely similar to that of Theorem 1.4 in \cite{CDQ}. We only need to replace Theorem 1.3 in \cite{CDQ} by Corollary \ref{cor1c} in the proof. Thus we omit the details of the proof.
\end{rem}

\begin{rem}\label{remark6c}
The lower bounds \eqref{lower-bound} on the $L^{\infty}$ norm of positive solutions $u$ indicate that, if $diam\,\Omega<\sqrt{2n}$, then a uniform priori estimate does not exist and blow-up may occur when $p\rightarrow1+$.
\end{rem}

\subsection{Extensions to general nonlinearities}

Consider the following integral equations associated with Navier problems for general critical order elliptic equations on $\mathbb{R}^{n}_{+}$:
\begin{equation}\label{eIE}
	u(x)=\int_{\mathbb{R}^{n}_{+}}G^{+}(x,y)f(y,u(y))dy,
\end{equation}
where $u\in C(\overline{\mathbb{R}^{n}_+})$, $n\geq1$ and the nonlinear terms $f:\,\mathbb{R}^{n}_{+}\times\overline{\mathbb{R}_{+}}\rightarrow \overline{\mathbb{R}_{+}}$.

\begin{defn}\label{defn1}
We say that the nonlinear term $f$ has subcritical growth, provided that
\begin{equation}\label{e1}
  \mu^{n}f(\mu x,u)
\end{equation}
is strictly increasing with respect to $\mu\geq1$ or $\mu\leq1$ for all $(x,u)\in\mathbb{R}^{n}_{+}\times\mathbb{R}_{+}$.
\end{defn}
\begin{defn}\label{defn2}
A function $g(x,u)$ defined on $\mathbb{R}^{n}_{+}\times\overline{\mathbb{R}_{+}}$ is called locally Lipschitz on $u$, provided that for any $u_{0}\in\overline{\mathbb{R}_{+}}$ and $\omega\subseteq\mathbb{R}^{n}_{+}$ bounded, there exists a (relatively) open neighborhood $U(u_{0})\subset\overline{\mathbb{R}_{+}}$ such that $g$ is Lipschitz continuous on $u$ in $\omega\times U(u_{0})$.
\end{defn}

We need the following three assumptions on the nonlinear term $f(x,u)$. \\
$(\mathbf{f_{1}})$ The nonlinear term $f$ is non-decreasing about $u$ in $\mathbb{R}^{n}_{+}\times\overline{\mathbb{R}_{+}}$, namely,
\begin{equation}\label{e2}
  (x,u), \, (x,v)\in\mathbb{R}^{n}_{+}\times\overline{\mathbb{R}_{+}} \,\,\, \text{with} \,\,\, u\leq v \,\,\, \text{implies} \,\,\, f(x,u)\leq f(x,v).
\end{equation}
$(\mathbf{f_{2}})$ There exists a $\sigma<n$ such that, $|x|^{\sigma}f(x,u)$ is locally Lipschitz on $u$ in $\mathbb{R}^{n}_{+}\times\overline{\mathbb{R}_{+}}$. \\
$(\mathbf{f_{3}})$ There exist a cone $\mathcal{C}\subset\mathbb{R}^{n}_{+}$ containing the positive $x_{n}$-axis with vertex at $0$ (say, $\mathcal{C}=\{x\in\mathbb{R}^{n}_{+}\,|\,x_{n}>\frac{|x|}{\sqrt{n}}\}$), constants $C>0$, $-n<a<+\infty$ and $0<p<+\infty$ such that, the nonlinear term
\begin{equation}\label{e3}
  f(x,u)\geq C|x|^{a}u^{p} \qquad \text{in} \,\,\, \mathcal{C}\times\overline{\mathbb{R}_{+}}.
\end{equation}

\smallskip

By applying the method of scaling spheres to the generalized integral equations \eqref{eIE}, we can derive the following Liouville theorem.
\begin{thm}\label{generalIE}
Assume $f$ is subcritical and satisfies the assumptions $(\mathbf{f_{1}})$, $(\mathbf{f_{2}})$ and $(\mathbf{f_{3}})$, then the Liouville type results in Theorem \ref{thmIE} are valid for integral equations \eqref{eIE}.
\end{thm}

\begin{rem}
By using the method of scaling spheres, Theorem \ref{generalIE} can be proved through a quite similar way as in the proof of Theorem \ref{thmIE}, so we leave the details to readers. We would like to mention that, if the nonlinear term $f(x,u)$ satisfies subcritical conditions for $\mu\leq1$ (see Definition \ref{defn1}), we only need to carry out calculations and estimates outside the upper half ball $B^{+}_{\lambda}(0)$ during the scaling spheres procedure.
\end{rem}

\begin{rem}
In particular, $f(x,u)=|x|^{a}u^{p}$ with $a>-n$ satisfies all the assumptions in Theorem \ref{generalIE}, thus Theorem \ref{thmIE} can also be regarded as a corollary of Theorem \ref{generalIE}. In addition, $f(x,u)=|x|^{a}(x_{n})^{b}u^{p}$ with $a+b>-n$ and $b\geq0$ also satisfies all the assumptions in Theorem \ref{generalIE}.
\end{rem}

Next, we consider the following Navier problems for general critical order elliptic equations on $\mathbb{R}^{n}_{+}$:
\begin{equation}\label{eNPDE}\\\begin{cases}
(-\Delta)^{\frac{n}{2}}u(x)=f(x,u(x)), \,\,\,\,\,\,\, u(x)\geq0, \,\,\,\,\,\,\,\, \,\,\, x\in\mathbb{R}^{n}_+, \\
u=(-\Delta)u=\cdots=(-\Delta)^{\frac{n}{2}-1}u=0 \,\,\,\,\,\,\,\,\,\,\, \text{on} \,\,\,  \partial\mathbb{R}^{n}_+,
\end{cases}\end{equation}
where $u\in C^{n}(\mathbb{R}^{n}_+)\cap C^{n-2}(\overline{\mathbb{R}^{n}_+})$, $n\geq2$ is even and the nonlinear terms $f:\,\mathbb{R}^{n}_{+}\times\overline{\mathbb{R}_{+}}\rightarrow \overline{\mathbb{R}_{+}}$.

It is clear from the proof of Theorem \ref{Thm0} that (see Section 2), under the same assumptions, the super poly-harmonic properties in Theorem \ref{Thm0} also hold for nonnegative classical solutions to the generalized critical order elliptic equations \eqref{eNPDE} provided that
\begin{equation}\label{Inequality}
f(x,u)\geq C|x|^{a}u^{p} \qquad \text{in} \,\,\, \mathbb{R}^{n}_{+}\times\overline{\mathbb{R}_{+}}.
\end{equation}
Based on the super poly-harmonic properties, one can verify under some assumptions on $f(x,u)$ that the proof of Theorem \ref{equivalence} can also be adopted to show the equivalence between the generalized critical order PDEs \eqref{eNPDE} and IEs \eqref{eIE} (see Section 3). For these purpose, we need the following assumptions on the nonlinear term $f(x,u)$. \\
$(\mathbf{f_{2}'})$ The nonlinear term $f(x,u)$ is locally Lipschitz on $u$ in $\mathbb{R}^{n}_{+}\times\overline{\mathbb{R}_{+}}$. \\
$(\mathbf{f_{3}'})$ There exist constants $C>0$, $0\leq a<+\infty$ and $0<p<+\infty$ such that, the nonlinear term $f(x,u)$ satisfies \eqref{Inequality}.

\smallskip

As a consequence of Theorem \ref{generalIE}, we derive the following Liouville theorem for the generalized critical order PDEs \eqref{eNPDE}.
\begin{thm}\label{generalNPDE}
Assume $f$ is subcritical and satisfies the assumptions $(\mathbf{f_{1}})$, $(\mathbf{f_{2}'})$ and $(\mathbf{f_{3}'})$, then the Liouville type results in Theorem \ref{thm1} are valid for PDEs \eqref{eNPDE}.
\end{thm}

\begin{rem}
In particular, $f(x,u)=|x|^{a}u^{p}$ with $a\geq0$ satisfies all the assumptions in Theorem \ref{generalNPDE}, thus Theorem \ref{thm1} can also be regarded as a corollary of Theorem \ref{generalNPDE}.
\end{rem}

The rest of this paper is organized as follows. In section 2, we prove the super poly-harmonic properties for nonnegative solutions to \eqref{NPDE} (i.e., Theorem \ref{Thm0}) via a variant of the method used in \cite{CFL}. In section 3, we show the equivalence between PDE \eqref{NPDE} and IE \eqref{IE}, namely, Theorem \ref{equivalence}. Section 4 is devoted to the proof of Theorem \ref{thmIE}, then Theorem \ref{thm1} follows immediately as a consequence of Theorem \ref{thmIE}.

In the following, we will use $C$ to denote a general positive constant that may depend on $n$, $a$, $p$ and $u$, and whose value may differ from line to line.

\section{Super poly-harmonic properties}

In this section, we will prove Theorem \ref{Thm0}. To this end, we make an odd extension of $u$ to the whole space $\mathbb{R}^{n}$. Define
\begin{equation}\label{5-1}
  u(x',x_n)=-u(x',-x_n) \quad\, \text{for} \,\,\, x_n<0,
\end{equation}
where $x'=(x_1,\cdots,x_{n-1})\in\mathbb{R}^{n-1}$. Then $u$ satisfies
\begin{equation}\label{5-2}
  (-\Delta)^{\frac{n}{2}}u=|x|^{a}|u|^{p-1}u(x), \quad\,\,\, x\in\mathbb{R}^{n}.
\end{equation}
Let $v_{i}:=(- \Delta)^{i}u$. We aim to show that
\begin{equation}\label{5-3}
  v_{i}(x)\geq0
\end{equation}
for any $x\in\overline{\mathbb{R}^{n}_{+}}$ and $i=1,2,\cdots,\frac{n}{2}-1$. Our proof will be divided into two steps.
	
\textbf{\emph{Step 1.}} We first show that
	\begin{equation}\label{2-1}
	v_{\frac{n}{2}-1}=(-\Delta)^{\frac{n}{2}-1}u\geq0.
	\end{equation}
	If \eqref{2-1} does not hold, then there exists $x_{0}\in\mathbb{R}^n_+$, such that
	\begin{equation}\label{2-2}
	v_{\frac{n}{2}-1}(x_{0})<0.
	\end{equation}
	
	Now, let
	\begin{equation}\label{2-3}
	\bar{f}(r)=\bar{f}\big(|x-x_0|\big):=\frac{1}{|\partial B_{r}(x_{0})|}\int_{\partial B_{r}(x_{0})}f(x)d\sigma
	\end{equation}
	be the spherical average of $f$ with respect to the center $x_0$. Then by the well-known property $\overline{\Delta u}=\Delta\bar{u}$, we have
	\begin{equation}\label{2-4}
	\left\{{\begin{array}{l} {-\Delta\overline{v_{\frac{n}{2}-1}}(r)=\overline{|x|^{a}|u|^{p-1}u}(r)}, \\  {} \\ {-\Delta\overline{v_{\frac{n}{2}-2}}(r)=\overline{v_{\frac{n}{2}-1}}(r)}, \\ \cdots\cdots \\ {-\Delta\overline u(r)=\overline{v_1}(r)}. \\ \end{array}}\right.
	\end{equation}
	From the first equation in \eqref{2-4}, integrating both sides from $0$ to $r$, we have
	\begin{align}\label{2-5}
	-r^{n-1}\overline{v_{\frac{n}{2}-1}}'(r)&=\int_{0}^{r} s^{n-1}\overline{|x|^{a}|u|^{p-1}u}(s)ds\nonumber\\
	&=\frac{1}{\omega_n}\int_0^r \int_{\partial B_{s}(x_0)}|x|^{a}|u|^{p-1}u\,d\sigma ds\\
	&=\frac{1}{\omega_n}\int_{B_{r}(x_0)}|x|^{a}|u|^{p-1}u\,dx\geq 0,\nonumber
	\end{align}
where $\omega_{n}$ denotes the area of unit sphere in $\mathbb{R}^{n}$. Here we have used the fact that, since $x_0\in\mathbb{R}^{n}_{+}$, more than half of the ball $B_{r}(x_0)$ is contained in $\mathbb{R}^n_+$, so \eqref{2-5} follows from the odd symmetry of $u$ with respect to $\partial \mathbb{R}^n_+$. From \eqref{2-2} and \eqref{2-5}, one has
	\begin{equation}\label{2-6}
	\overline{v_{\frac{n}{2}-1}}'(r)\leq 0, \,\,\,\,\,\, \overline{v_{\frac{n}{2}-1}}(r)\leq\overline{v_{\frac{n}{2}-1}}(0)=v_{\frac{n}{2}-1}(x_0)<0, \,\,\,\,\,\, \forall \, r\geq 0.
	\end{equation}
	Then from the second equation in \eqref{2-4}, we have
	\begin{equation}\label{2-7}
	-\frac{1}{r^{n-1}}(r^{n-1}\overline{v_{\frac{n}{2}-2}}'(r))'= \overline{v_{\frac{n}{2}-1}}(r)\leq\overline{v_{\frac{n}{2}-1}}(0):=-c_0<0,\,\,\,\,\,\,\forall \, r\geq 0,
	\end{equation}
	which means
\begin{equation}\label{5-4}
  (r^{n-1}\overline{v_{\frac{n}{2}-2}}'(r))'\geq c_{0}r^{n-1},\,\,\,\,\,\,\forall \, r\geq 0.
\end{equation}
   Integrating from $0$ to $r$ twice yields
	\begin{equation}\label{2-8}
	\overline{v_{\frac{n}{2}-2}}(r)\geq \frac{c_0}{2n}r^2+\overline{v_{\frac{n}{2}-2}}(0)\geq \frac{c_0}{2n}r^2+c_1,\,\,\,\,\,\,\forall \, r\geq 0.
	\end{equation}
	Continuing this way, if $\frac{n}{2}$ is odd, we can derive that
	\begin{equation}\label{2-9}
	\bar{u}(r)\leq -\tilde{c}_{0}r^{n-2}+\sum_{i=1}^{\frac{n}{2}-1}\tilde{c}_{i}r^{2(\frac{n}{2}-1-i)},\,\,\,\,\,\,\forall \, r\geq 0,
	\end{equation}
	where $\tilde{c}_0>0$ and $\tilde{c}_{\frac{n}{2}-1}=\bar{u}(0)=u(x_{0})\geq0$. Then, similar to \eqref{2-5}, by the definition of $\bar{u}$ and \eqref{2-9}, for $r$ large, we obtain
	\begin{eqnarray}\label{2-10}
	&&0\leq \frac{1}{\omega_n}\int_{B_r{(x_0)}}udx=\int_0^r s^{n-1}\bar{u}(s)ds \\
\nonumber &\leq&-\frac{\tilde{c}_{0}}{2n-2}r^{2n-2}+\sum_{i=1}^{\frac{n}{2}-1}\frac{\tilde{c}_{i}}{2(n-1-i)}r^{2(\frac{n}{2}-1-i)+n}<0,
	\end{eqnarray}
	which is absurd. Hence in the following, we assume that $\frac{n}{2}$ is even.
	
	Set $w_0(r)=\bar{u}(r)$ and $w_k(r)=\Delta^k\bar{u}(r)$, $k=1,\cdots,\frac{n}{2}-1$.
	By \eqref{2-2}, we have $w_{\frac{n}{2}-1}(0)=-\overline{v_{\frac{n}{2}-1}}(0)>0$ and $w_k$ satisfies
	\begin{equation}\label{2-11}
	\begin{cases}
	\Delta w_{\frac{n}{2}-1}(r)=\overline{|x|^{a}|u|^{p-1}u}(r),\\
	\Delta w_k(r)=w_{k+1}(r),\,\,\,\,\,\,k=0,\cdots,\frac{n}{2}-2.
	\end{cases}
	\end{equation}
	We divide \eqref{2-11} into two parts. Let $w_k:=u_k+\phi_k$, $k=0,\cdots,\frac{n}{2}-1$, where $u_k$ satisfies
	\begin{equation}\label{2-12}
	\begin{cases}
	\Delta u_{\frac{n}{2}-1}(r)=\overline{|x|^{a}|u|^{p-1}u}(r),\\
	\Delta u_k(r)=u_{k+1}(r),\,\,\,\,\,\,k=0,\cdots,\frac{n}{2}-2,\\
	u_k(0)=0,\,\,\,\,\,\,k=0,\cdots,\frac{n}{2}-1,
	\end{cases}
	\end{equation}
	and $\phi_k$ solves
	\begin{equation}\label{2-13}
	\begin{cases}
	\Delta \phi_{\frac{n}{2}-1}(r)=0,\\
	\Delta \phi_k(r)=\phi_{k+1}(r),\,\,\,\,\,\,k=0,\cdots,\frac{n}{2}-2,\\
	\phi_k(0)=w_k(0),\,\,\,\,\,\,k=0,\cdots,\frac{n}{2}-1.
	\end{cases}
	\end{equation}
	From \eqref{2-13}, by direct calculations, we have
	\begin{equation}\label{2-14}
	\phi_0(r)=\sum_{k=0}^{\frac{n}{2}-1}c_kw_k(0)r^{2k},
	\end{equation}
	where $c_k>0,\,\,\,k=0,\cdots,\frac{n}{2}-1$. It is easy to see that
\begin{equation}\label{5-5}
  \phi_0(r)\geq c_0\bar{u}(0)+c_{\frac{n}{2}-1}w_{\frac{n}{2}-1}(0)r^{n-2}-\sum_{k=1}^{\frac{n}{2}-2}c_k|w_k(0)|r^{2k}.
\end{equation}
Let
\begin{equation}\label{5-6}
  u_\lambda(x)=\lambda^{\frac{n+a}{p-1}}u(\lambda x)
\end{equation}
be the re-scaling of $u$. Then one can verify that $u_\lambda$ still satisfies the equation
\begin{equation}\label{5-7}
  \Delta^{\frac{n}{2}} u_\lambda=|x|^a|u_\lambda|^{p-1}u_\lambda, \,\,\quad\, x\in\mathbb{R}^n.
\end{equation}
Let $w_{0,\lambda}(r)=\overline{u_\lambda}(r)\,\text{and}\,w_{k,\lambda}(r)=\Delta^k\overline{u_\lambda}(r),\,\,\,k=1,\cdots,\frac{n}{2}-1$, where the spherical average is taken with respect to the center $x_{0,\lambda}:=\frac{x_0}{\lambda}$, we have
\begin{equation}\label{5-8}
  w_{k,\lambda}(0)=\lambda^{\frac{n+a}{p-1}+2k}w_k(0).
\end{equation}
Similar to $w_k$, we decompose $w_{k,\lambda}:=u_{k,\lambda}+\phi_{k,\lambda},\,\,\,k=0,\cdots,\frac{n}{2}-1$, where $u_{k,\lambda}$ and $\phi_{k,\lambda}$ still satisfy \eqref{2-12} and \eqref{2-13} respectively if we substitute $u_\lambda$ and $w_{k,\lambda}$ for $u$ and $w_k$. Similar to \eqref{5-5}, we can still conclude that
\begin{equation}\label{2-15}
	\phi_{0,\lambda}(r)\geq \lambda^{\frac{n+a}{p-1}}\left(c_0\bar{u}(0)+c_{\frac{n}{2}-1}w_{\frac{n}{2}-1}(0)\lambda^{n-2}r^{n-2}-\sum_{k=1}^{\frac{n}{2}-2}c_k|w_k(0)|\lambda^{2k}r^{2k}\right).
\end{equation}
	
Next we use the iteration argument to derive a contradiction.
	
Set $\Omega_{\tau}^+=B_\tau(x_{0,\lambda})\cap\mathbb{R}_+^n,\Omega_{\tau}^-=B_\tau(x_{0,\lambda})\cap\left(\mathbb{R}^n\setminus\mathbb{R}_+^n\right)$. Let $\widetilde{\Omega_{\tau}^-}$ be the reflection of $\Omega_{\tau}^-$ with respect to $\partial\mathbb{R}_+^n$ and $\Omega_\tau=\Omega_{\tau}^+\setminus\widetilde{\Omega_{\tau}^-}$.
	
We can choose $\lambda$ large such that $|x_{0,\lambda}|<\frac{1}{4}$, then it is easy to see that given $1\leq \tau \leq2$, for all $x\in \Omega_\tau$, we have $|x|^{a}\geq(1-|x_{0,\lambda}|)^a>(\frac{3}{4})^a$ if $a\geq0$, and $|x|^{a}\geq(2+|x_{0,\lambda}|)^a>(\frac{9}{4})^a$ if $a<0$. By the first equation of \eqref{2-12}, Jensen's inequality and odd symmetry of $u_\lambda$ with respect to $\partial \mathbb{R}^n_+$, we have, for any $1\leq r\leq 2$,
	\begin{align}\label{2-16}
	u_{\frac{n}{2}-1,\lambda}(r)&=\int_0^r \frac{1}{\tau^{n-1}}\int_0^{\tau}s^{n-1}\overline{|x|^{a}|u_\lambda|^{p-1}u_\lambda}dsd\tau\\
	&=\int_0^r \frac{1}{\tau^{n-1}}\int_0^{\tau}s^{n-1}\frac{1}{\partial B_{s}(x_{0,\lambda})}\int_{\partial B_{s}(x_{0,\lambda})}|x|^a|u_\lambda|^{p-1}u_\lambda(x) d\sigma ds d\tau\nonumber\\
	&=\frac{1}{\omega_n}\int_0^r \frac{1}{\tau^{n-1}}\int_{B_\tau(x_{0,\lambda})}|x|^{a}|u_\lambda|^{p-1}u_\lambda(x)dxd\tau\nonumber\\
	&=\frac{1}{\omega_n}\int_0^r \frac{1}{\tau^{n-1}}\int_{\Omega_\tau}|x|^{a}u_\lambda^{p}(x)dxd\tau\geq C'_0\int_1^r\frac{1}{\tau^{n-1}}\int_{\Omega_\tau}u_\lambda^{p}(x)dxd\tau\nonumber \\
	&=C'_0\int_1^r\frac{|\Omega_\tau|}{\tau^{n-1}}\left(\frac{1}{|\Omega_\tau|}\int_{\Omega_\tau}u_\lambda^{p}(x)dx\right)d\tau\nonumber\\
	&\geq C'_0\int_1^r\frac{|\Omega_\tau|}{\tau^{n-1}}\left(\frac{1}{|\Omega_\tau|}\int_{\Omega_\tau}u_\lambda(x)dx\right)^{p}d\tau\nonumber\\
	&=C'_0\int_1^r\frac{1}{|B_\tau(x_{0,\lambda})|^{p-1}\tau^{n-1}}\left(\frac{|B_\tau(x_{0,\lambda})|}{|\Omega_\tau|}\right)^{p-1}
\left(\int_{\Omega_\tau}u_\lambda(x)dx\right)^{p}d\tau\nonumber\\
	&\geq C''_0\int_1^r\frac{1}{\tau^{np-p}}\left(\int_{B_\tau(x_{0,\lambda})}u_\lambda(x)dx\right)^{p}d\tau\nonumber\\
	&= C''_0\int_1^r\frac{1}{\tau^{np-p}}\left(\int_0^\tau \int_{\partial B_s(x_{0,\lambda})}u_\lambda(x)d\sigma ds\right)^{p}d\tau\nonumber\\
	&=: C_0\int_1^r\frac{1}{\tau^{np-p}}\left(\int_0^\tau \overline{u_\lambda}(s)s^{n-1} ds\right)^{p}d\tau,\nonumber
	\end{align}
where $C_0\in(0,1]$ is a positive constant depending only on $n, \,p \, \text{and} \, a$ if $\lambda$ is large enough.
	
It follows from \eqref{2-15} and $\frac{n+a}{p-1}+n-2>0$ that, we can choose $\lambda$ sufficiently large to make $\phi_{0,\lambda}(r)$ as large as we wish for $1\leq r \leq 2$. Apparently, $u_{0,\lambda}(r)\geq 0$, from $\overline{u_\lambda}=u_{0,\lambda}+\phi_{0,\lambda}$ and \eqref{2-15}, we can choose $\lambda$ sufficiently large such that
	\begin{equation}\label{2-17}
	\overline{u_\lambda}(r)\geq a_0(r-1)^{\sigma_0},\,\,\,\,\,\, \forall \, 1\leq r \leq 2,
	\end{equation}
	where $a_0$ and $\sigma_0$ are arbitrarily large, and will be determined later. By elementary calculation, it is easy to verify that
	\begin{equation}\label{2-18}
	\int_1^\tau (s-1)^{\alpha}s^{\beta}ds\geq \frac{1}{\alpha+\beta+1}(\tau-1)^{\alpha+1}{\tau}^\beta,\,\,\,\,\,\,\,\,\,\forall \,\alpha,\,\beta>0.
	\end{equation}
	By \eqref{2-16}, \eqref{2-17} and \eqref{2-18}, for any $1\leq r \leq2$ and $p>1$, we obtain
	\begin{align}\label{2-19}
	u_{\frac{n}{2}-1,\lambda}(r)&\geq C_0\int_1^r\frac{1}{\tau^{p(n-1)}}\left(\int_0^\tau \overline{u_\lambda}(s)s^{n-1} ds\right)^{p}d\tau\\
	&\geq  C_0\int_1^r\frac{1}{\tau^{p(n-1)}}\left(\int_1^\tau a_0(s-1)^{\sigma_0}s^{n-1} ds\right)^{p}d\tau\nonumber\\
	&\geq  C_0\int_1^r\frac{1}{\tau^{p(n-1)}}\left(\frac{a_0}{\sigma_0+n}(\tau-1)^{\sigma_0+1}\tau^{n-1}\right)^{p}d\tau\nonumber\\
	&\geq \frac{C_0a_0^p}{(\sigma_0+n)^p} \int_1^r (\tau-1)^{(\sigma_0+1)p}d\tau\nonumber\\
	&= \frac{C_0a_0^p}{(\sigma_0+n)^p[(\sigma_0+1)p+1]} (r-1)^{(\sigma_0+1)p+1}.\nonumber
	\end{align}
	Set $\sigma_0$ large such that $\sigma_0\geq p+2n$, then \eqref{2-19} implies that
	\begin{equation}\label{2-20}
	u_{\frac{n}{2}-1,\lambda}(r) \geq \frac{C_0a_0^p}{(2\sigma_0)^p(2\sigma_0p)} (r-1)^{(\sigma_0+1)p+1}.
	\end{equation}
	From the second equation in \eqref{2-12} and \eqref{2-20}, we have
	\begin{equation}\label{2-21}
	(r^{n-1}u_{\frac{n}{2}-2,\lambda}')'(r)=r^{n-1}u_{\frac{n}{2}-1,\lambda}(r)\geq \frac{C_0a_0^p}{(2\sigma_0)^p(2\sigma_0p)} (r-1)^{(\sigma_0+1)p+1}r^{n-1},\,\,\,\,\,\,\forall \, 1\leq r\leq 2.
	\end{equation}
	Since $(r^{n-1}u_{\frac{n}{2}-2,\lambda}')'(r)=r^{n-1}u_{\frac{n}{2}-1,\lambda}(r)\geq0$ for any $r\geq 0$, by \eqref{2-21}, we derive
	\begin{align}\label{2-22}
	r^{n-1}u_{\frac{n}{2}-2,\lambda}'(r)&=\int_0^r (\tau^{n-1}u'_{\frac{n}{2}-2,\lambda}(\tau))'d\tau\nonumber\\
	&\geq \frac{C_0a_0^p}{(2\sigma_0)^p(2\sigma_0p)}\int_1^r (\tau-1)^{(\sigma_0+1)p+1}\tau^{n-1}d\tau\\
	&\geq \frac{C_0a_0^p}{(2\sigma_0)^p(2\sigma_0p)[(\sigma_0+1)p+n+1]}(r-1)^{(\sigma_0+1)p+2}r^{n-1},\,\,\,\,\,\,\,\forall \, 1\leq r \leq 2,\nonumber
	\end{align}
	which means
	\begin{equation}\label{2-23}
	u_{\frac{n}{2}-2,\lambda}'(r)\geq \frac{C_0a_0^p}{(2\sigma_0)^p(2\sigma_0p)[(\sigma_0+1)p+n+1]}(r-1)^{(\sigma_0+1)p+2}.
	\end{equation}
	By $(r^{n-1}u_{\frac{n}{2}-2,\lambda}')'(r)=r^{n-1}u_{\frac{n}{2}-1,\lambda}\geq0$ for any $r\geq 0$, we have
\begin{equation}\label{5-9}
  r^{n-1}u_{\frac{n}{2}-2,\lambda}'(r)=\int_0^r \tau^{n-1}u_{\frac{n}{2}-1,\lambda}(\tau)d\tau\geq0,\,\,\,\forall \, r\geq 0,
\end{equation}
thus $u_{\frac{n}{2}-2,\lambda}'(r)\geq0, \,\,\,\forall \, r\geq 0$. Similar to \eqref{2-22}, by \eqref{2-23}, we have
	\begin{align}\label{2-24}
	u_{\frac{n}{2}-2,\lambda}(r)&=\int_0^r u_{\frac{n}{2}-2,\lambda}'(\tau)d\tau\nonumber\\
	&\geq \int_1^r u_{\frac{n}{2}-2,\lambda}'(\tau)d\tau\nonumber\\
	&\geq \frac{C_0a_0^p}{(2\sigma_0)^p(2\sigma_0p)[(\sigma_0+1)p+n+1]}\int_1^r (\tau-1)^{(\sigma_0+1)p+2}d\tau\\
	&\geq \frac{C_0a_0^p}{(2\sigma_0)^p(2\sigma_0p)[(\sigma_0+1)p+n+1][(\sigma_0+1)p+3]}(r-1)^{(\sigma_0+1)p+3}\nonumber\\
	&\geq \frac{C_0a_0^p}{(2\sigma_0)^p(2\sigma_0p)^3}(r-1)^{(\sigma_0+1)p+3},\,\,\,\,\,\,\forall \, 1\leq r \leq 2.\nonumber
	\end{align}
	
	Continuing this way, we eventually obtain that
	\begin{equation}\label{2-25}
	u_{0,\lambda}(r)\geq \frac{C_0a_0^p}{(2\sigma_0)^p(2\sigma_0p)^{n-1}}(r-1)^{(\sigma_0+1)p+n-1},\,\,\,\quad \forall \, 1\leq r \leq 2.
	\end{equation}
	Thus, we have
	\begin{equation}\label{2-26}
	\overline{u_{\lambda}}(r)\geq \frac{C_0a_0^p}{(2\sigma_0)^p(2\sigma_0p)^{n-1}}(r-1)^{(\sigma_0+1)p+n-1},\,\,\,\quad \forall \, 1\leq r \leq 2.
	\end{equation}
	
	We set $\sigma_1:=2p\sigma_0$ and $a_1=\frac{C_0a_0^p}{(2\sigma_0p)^{n-1+p}}$. Then, for any $1\leq r \leq 2$, by \eqref{2-26}, we have
	\begin{equation*}
	\overline{u_{\lambda}}(r)\geq a_1(r-1)^{\sigma_1},\,\,\,\,\, \forall \, 1\leq r\leq 2.
	\end{equation*}
	
	Repeating the above arguments, we have
	\begin{equation}\label{2-27}
	\overline{u_{\lambda}}(r)\geq a_k(r-1)^{\sigma_k},\,\,\,\,\, \forall \, 1\leq r\leq 2,
	\end{equation}
	where $\sigma_k=2p\sigma_{k-1}$, $a_k=\frac{C_0a_{k-1}^p}{(2\sigma_{k-1}p)^{n-1+p}}$ and $k=2,3,\cdots$.
	
	We can prove that $a_k\rightarrow +\infty$ as $k\rightarrow+\infty$. In fact, by direct calculations, we have
	\begin{align}\label{2-28}
	a_k&=\frac{C_0^{\frac{p^k-1}{p-1}}a_0^{p^k}}{(2p)^{(n-1+p)(k+(k-1)p+\cdots+p^{k-1})}\sigma_0^{(n-1+p)\frac{p^k-1}{p-1}}}\\
&=C_0^{\frac{-1}{p-1}}(2p)^{(n-1+p)\left(\frac{p}{(p-1)^2}+\frac{k}{p-1}\right)}\sigma_0^{\frac{n-1+p}{p-1}}\left( \frac{C_0^{\frac{1}{p-1}}a_0}{(2p)^{\frac{(n-1+p)p}{(p-1)^2}}\sigma_0^{\frac{n-1+p}{p-1}}}\right)^{p^k}. \nonumber
	\end{align}
	Take $a_0=2C_0^{-\frac{1}{p-1}}(2p)^{\frac{(n-1+p)p}{(p-1)^2}}\sigma_0^{\frac{n-1+p}{p-1}}$, then by \eqref{2-27} and \eqref{2-28}, we can see that
	\begin{equation*}
	\overline{u_{\lambda}}(2)\geq a_k \geq (2p)^{\frac{(n-1+p)k}{p-1}}2^{p^k} \rightarrow +\infty,\,\,\,\,\,\,\text{as}\,\, k\rightarrow\infty,
	\end{equation*}
	which is absurd! Thus \eqref{2-1} must hold, that is, $(-\Delta)^{\frac{n}{2}-1}u\geq 0$ in $\overline{\mathbb{R}^{n}_{+}}$.
	
	\textbf{\emph{Step 2.}}
	Next, we will show that all the other $v_k(x)\geq 0,\,k=1,2,\cdots,\frac{n}{2}-2,\,\forall x\in\overline{\mathbb{R}^{n}_+}$. Suppose not, then there exists some $2\leq i\leq \frac{n}{2}-1$ and $x_0 \in \mathbb{R}^{n}_+$ such that
	\begin{gather}
	v_{\frac{n}{2}-1}(x)\geq 0,\,\,\,\,v_{\frac{n}{2}-2}(x)\geq 0,\,\,\,\, \cdots,\,\,\,\, v_{\frac{n}{2}-i+1}(x)\geq 0, \quad\,\, \forall\, x\in\mathbb{R}^{n}_+,\label{2-29}\\
	v_{\frac{n}{2}-i}(x_0)<0.\label{2-30}
	\end{gather}
	
	Take spherical average with respect to the center $x_0$, we have $\overline{v_{\frac{n}{2}-i}}$ satisfies
	\begin{equation}\label{2-31}
	\left\{{\begin{array}{l} {-\Delta\overline{v_{\frac{n}{2}-i}}(r)=\overline{v_{\frac{n}{2}-i+1}}}, \\  {} \\ {-\Delta\overline{v_{\frac{n}{2}-i-1}}(r)=\overline{v_{\frac{n}{2}-i}}(r)}, \\ \cdots\cdots \\ {-\Delta\overline u(r)=\overline{v_1}(r)}. \\ \end{array}}\right.
	\end{equation}
	By the first equation of \eqref{2-31}, integrating both sides from $0$ to $r$, we arrive at
	\begin{align}\label{2-32}
	-r^{n-1}\overline{v_{\frac{n}{2}-i}}'(r)&=\int_{0}^{r} s^{n-1}\overline{v_{\frac{n}{2}-i+1}}ds\nonumber\\
	&=\frac{1}{\omega_n}\int_0^r \int_{\partial B_{s}(x_0)}v_{\frac{n}{2}-i+1}\,d\sigma ds\\
	&=\frac{1}{\omega_n}\int_{B_{r}(x_0)}v_{\frac{n}{2}-i+1}\,dx\geq 0.\nonumber
	\end{align}
	Here we have used the fact that, since $x_0\in\mathbb{R}^{n}_{+}$, more than half of the ball $B_{r}(x_0)$ is contained in $\mathbb{R}^n_+$, so \eqref{2-32} follows from the odd symmetry of $v_{\frac{n}{2}-i+1}$ with respect to $\partial \mathbb{R}^n_+$.
	
As in Step 1, we can also derive a contradiction if $\frac{n}{2}-i$ is even, hence we assume that $\frac{n}{2}-i$ is odd hereafter. By the same arguments as in deriving \eqref{2-9} in Step 1, we can obtain that
	\begin{equation}\label{2-33}
	\bar{u}(r)\geq c_{0}r^{2(\frac{n}{2}-i)}+\sum_{j=1}^{\frac{n}{2}-i}c_{j}r^{2(\frac{n}{2}-i-j)},\,\,\,\quad \forall \, r\geq 0,
	\end{equation}
where $c_{0}>0$ and $c_{\frac{n}{2}-i}=\bar{u}(0)=u(x_{0})\geq0$. Therefore, if we assume that $u(x)=o(|x|^{2})$ as $|x|\rightarrow+\infty$, we will get a contradiction from \eqref{2-33} immediately. We only need to discuss under the assumptions $-2p-2\leq a<+\infty$ hereafter.
	
Notice that there exists $r_0$ large enough such that if $r\geq r_0$, then $\forall x\in \Omega_r$, $|x|^a\geq (r-|x_0|)^a\geq Cr^a$ if $a\geq 0$, and  $|x|^a\geq (r+|x_0|)^a\geq Cr^a$ if $a<0$, furthermore, by \eqref{2-33}, we can also get, if $r\geq r_0$,
\begin{equation}\label{2-33'}
	\bar{u}(r)\geq \frac{c_{0}}{2}r^{2(\frac{n}{2}-i)}.
\end{equation}
Similar to \eqref{2-16}, by \eqref{2-5}, Jensen's inequality and odd symmetry of $u$ with respect to $\partial \mathbb{R}^n_+$, we have for any $r\geq r_0$,
	\begin{align}\label{2-34}
	-r^{n-1}\overline{v_{\frac{n}{2}-1}}'(r)
	&=\frac{1}{\omega_n}\int_{B_{r}(x_0)}|x|^{a}|u|^{p-1}u\,dx\\
	&=\frac{1}{\omega_n}\int_{\Omega_r} |x|^{a}u^{p}\,dx\geq \frac{Cr^a}{\omega_n}\int_{\Omega_r} u^{p}\,dx\nonumber\\
	&\geq \frac{C|\Omega_r|r^a}{\omega_n}\left( \frac{1}{|\Omega_r|} \int_{\Omega_r} u^{p}\,dx\right)\nonumber\\
	&\geq \frac{C|\Omega_r|r^a}{\omega_n}\left( \frac{1}{|\Omega_r|} \int_{\Omega_r} u\,dx\right)^{p}\nonumber\\
	&= Cr^a\left(\frac{|B_r(x_0)|}{|\Omega_r|}\right)^{p-1} \frac{1}{\omega_n|B_r(x_0)|^{p-1}}\left(\int_{B_r(x_0)} u\,dx\right)^{p}\nonumber\\
	&\geq  \frac{Cr^a}{\omega_n|B_r(x_0)|^{p-1}}\left(\int_{B_r(x_0)} u\,dx\right)^{p}\nonumber\\
	&=\frac{Cr^a}{\omega_n|B_r(x_0)|^{p-1}}\left(\int_0^r \int_{\partial B_s(x_0)} u\,d\sigma\,ds \right)^{p}\nonumber\\
	&=\frac{C\omega_n^{p-1}r^a}{|B_r(x_0)|^{p-1}}\left(\int_0^r \bar{u}(s)s^{n-1}\,ds \right)^{p}\nonumber\\
	&=\frac{C}{r^{np-n-a}}\left(\int_0^r \bar{u}(s)s^{n-1}\,ds \right)^{p}.\nonumber
	\end{align}
	Combining \eqref{2-34} with \eqref{2-33'}, we have, for all $r\geq 2r_0$,
	\begin{eqnarray}\label{2-35}
	-r^{n-1}\overline{v_{\frac{n}{2}-1}}'(r)
	&\geq&\frac{C}{r^{np-n-a}}\left(\int_0^r \bar{u}(s)s^{n-1}\,ds \right)^{p}\\
	&\geq& \frac{C}{r^{np-n-a}}\left(\int_{r_0}^r s^{2(\frac{n}{2}-i)}s^{n-1}\,ds \right)^{p}\nonumber\\
	&\geq& \frac{C}{r^{np-n-a}}r^{np+2(\frac{n}{2}-i)p}\nonumber\\
	&\geq& Cr^{n+2(\frac{n}{2}-i)p+a}.\nonumber
	\end{eqnarray}
	That is,
	\begin{equation}\label{2-36}
	\overline{v_{\frac{n}{2}-1}}'(r)\leq -Cr^{2(\frac{n}{2}-i)p+a+1}.
	\end{equation}
	Integrating \eqref{2-36} in both sides from fixed $r_1\geq 2r_0$ to $r$, we obtain, if $a>-2-2p$, then
	\begin{equation}\label{2-37}
	\overline{v_{\frac{n}{2}-1}}(r)\leq -C\left( r^{2p+a+2}-r_{1}^{2p+a+2}\right)+\overline{v_{\frac{n}{2}-1}}(r_1)\rightarrow-\infty, \,\,\,\,\,\, \text{as} \,\,\, r\rightarrow+\infty;
	\end{equation}
	if $a=-2-2p$, then
	\begin{equation}\label{2-37'}
	\overline{v_{\frac{n}{2}-1}}(r)\leq -C\left( \ln r-\ln r_{1}\right)+\overline{v_{\frac{n}{2}-1}}(r_1)\rightarrow-\infty, \,\,\,\,\,\, \text{as} \,\,\, r\rightarrow+\infty.
	\end{equation}
However, from the proven fact that $v_{\frac{n}{2}-1}(x)\geq 0, x\in\overline{\mathbb{R}_+^n}$, the odd symmetry of $v_{\frac{n}{2}-1}$ with respect to $\partial \mathbb{R}^{n}_+$ and \eqref{2-37}, \eqref{2-37'}, we get that
	\begin{eqnarray}\label{2-38}
	0&\leq& \frac{1}{\omega_n}\int_{B_{r}{(x_0)}} v_{\frac{n}{2}-1}(x)dx \\
	&=&\frac{1}{\omega_n}\int_{B_{r_1}{(x_0)}} v_{\frac{n}{2}-1}(x)dx+\frac{1}{\omega_n}\int_{{B_{r}{(x_0)}}\setminus B_{r_1}{(x_0)}} v_{\frac{n}{2}-1}(x)dx\nonumber\\
	&=&\frac{1}{\omega_n}\int_{B_{r_1}{(x_0)}} v_{\frac{n}{2}-1}(x)dx+\int_{r_1}^{r}s^{n-1}\overline{v_{\frac{n}{2}-1}}(s)ds\rightarrow-\infty, \,\,\,\,\,\, \text{as} \,\,\, r\rightarrow+\infty. \nonumber
	\end{eqnarray}
This is a contradiction! Therefore, we arrive at $v_k(x)\geq 0,\,k=1,2,\cdots,\frac{n}{2}-2,\,\forall \,x\in\overline{\mathbb{R}^{n}_+}$.

This finishes our proof of Theorem \ref{Thm0}.

\section{Equivalence between PDE and IE}

In this section, we prove the equivalence between PDE \eqref{NPDE} and IE \eqref{IE}, namely, Theorem \ref{equivalence}. We only need to prove that any nonnegative solution of PDE \eqref{NPDE} also satisfies IE \eqref{IE}.

\emph{(i)} We first consider the cases that $n\geq4$ is even.

In section 2, we have proved that $v_{i}=(- \Delta)^{i}u\geq 0$ for $i=1,2,\cdots,\frac{n}{2}-1$, then \eqref{NPDE} is equivalent to the following system
\begin{equation}\label{PDES}
\left\{{\begin{array}{l} {-\Delta v_{\frac{n}{2}-1}(x)=|x|^{a}|u|^{p}(x)},  \\
	{-\Delta v_{\frac{n}{2}-2}(x)= v_{\frac{n}{2}-1}(x)}, \\ \cdots\cdots \\ {-\Delta u(x)= v_1(x)}. \\ \end{array}}\right.
\end{equation}
In the following, similar as in \cite{CFL}, we define
\begin{equation}\label{6-1}
  G(x,y,i):=c_n\left(\frac{1}{|x-y|^{n-i}}-\frac{1}{|\bar {x}-y|^{n-i}}\right),\,\,\,\,\, x,y\in \mathbb{R}^n_+,\,\,\,i=2,4,\ldots,n-2,
\end{equation}
and $G^+(x,y):=G(x,y,n)=C_{n}\left(\ln{\frac{1}{|x-y|}}-\ln{\frac{1}{|\bar{x}-y|}}\right)$.

Let $B_R^+=B_R(0)\cap R_+^n$. In \cite{CC}, Cao and Chen derived the Green's function $G_R(x,y,2)$ associated with $-\Delta$ for the half ball $B_R^+$, that is,
\begin{multline*}
G_R(x,y,2):=\frac{c_n}{|x-y|^{n-2}}-\frac{c_n}{\left(|x-y|^2+(R-\frac{|x|^2}{R})(R-\frac{|y|^2}{R})\right)^{\frac{n-2}{2}}}\\
-\left(\frac{c_n}{|\bar{x}-y|^{n-2}}-\frac{c_n}{\left(|\bar{x}-y|^2+(R-\frac{|x|^2}{R})(R-\frac{|y|^2}{R})\right)^{\frac{n-2}{2}}}\right).
\end{multline*}
Now we list some essential properties of the above Green's functions, which have been proved in \cite{CC} and \cite{CFL}. First,
\begin{equation}\label{3-4}
G_R(x,y,2)\rightarrow G(x,y,2)=\frac{c_n}{|x-y|^{n-2}}-\frac{c_n}{|\bar{x}-y|^{n-2}}, \,\,\,\,\,\,\text{as}\,\,R\rightarrow+\infty.
\end{equation}
Second, let $\Gamma_R$ be the hemisphere part of $\partial B_R^+$, then for each fixed $x \in B_R^+$ and any $y \in \Gamma_R$,
\begin{equation}\label{3-5}
\frac{\partial G_R(x,y,2)}{\partial\nu_y}=(2-n)R\left(1-\frac{|x|^2}{R^2}\right)\left(\frac{1}{|x-y|^{n}}-\frac{1}{|\bar{x}-y|^{n}} \right)<0,
\end{equation}
moreover, for each fixed $x \in B_R^+$, any $y \in \Gamma_R$ and $R$ sufficiently large,
\begin{gather}
G(x,y,2)=\frac{c_n}{|x-y|^{n-2}}-\frac{c_n}{|\bar{x}-y|^{n-2}} \sim \frac{y_n}{R^{n}}, \quad \text{as} \,\,\, R\rightarrow+\infty,\label{3-5-1}\\
\left|\frac{\partial G_R(x,y,2)}{\partial\nu_y}\right|=(n-2)R\left(1-\frac{|x|^2}{R^2}\right)\left(\frac{1}{|x-y|^{n}}-\frac{1}{|\bar{x}-y|^{n}} \right) \sim \frac{y_n}{R^{n+1}}.\label{3-5-2}
\end{gather}

In \cite{CFL}, the authors prove the following property of Green's functions with different orders on a half space:
\begin{equation}\label{3-6}
G(x,y,2k)=\int_{\mathbb{R}_+^n} G(x,z,2)G(z,y,2(k-1))dz,\,\,\,\quad \forall \, k=2,\cdots,\frac{n}{2}-1.
\end{equation}

To continue, we need to prove that the above property is still valid for the critical order $k=\frac{n}{2}$, this is the following crucial lemma.
\begin{lem}\label{lem3-0}
	Assume that $n\geq 4$ is even, then
	\begin{equation}\label{3-6-1}
	G^+(x,y)=\int_{\mathbb{R}_+^n} G(x,z,2)G(z,y,n-2)dz.
	\end{equation}
\end{lem}
\begin{proof}
By elementary calculations, one can easily verify that
\begin{gather}
-\Delta_y G^+(x,y)=G(x,y,n-2),\label{3-6-2}\\
G^+(x,y)=0,\,\,\,\,\,\,\, \text{if}\,\, \,x \,\,\, \text{or} \,\,\, y \in \partial\mathbb{R}_+^n.\label{3-6-3}
\end{gather}

For fixed $x,\,z \in B_R^+$, $z\neq x$ and $\epsilon$ sufficiently small, multiplying both side of \eqref{3-6-2} by $G_R(y,z,2)$ and integrating on $B_R^+\setminus B_{\epsilon}(z)$ by parts, we have
\begin{eqnarray}\label{3-6-4}
\text{LHS}&=& \int_{B_R^+\setminus B_{\epsilon}(z)} (-\Delta_y G^+(x,y))G_R(y,z,2)dy\\
&=& \int_{B_R^+\setminus B_{\epsilon}(z)} G^+(x,y)\left(-\Delta_yG_R(y,z,2)\right)dy\nonumber\\
&\,&\,+\int_{\partial B_{\epsilon}(z)} \left\{-\frac{\partial G^+(x,y)}{\partial \nu_y}G_R(y,z,2)+\frac{\partial G_R}{\partial \nu_y}(y,z,2)G^+(x,y)\right\}d\sigma_y\nonumber\\
&\,&\,+\int_{\partial B_R^+} \left\{-\frac{\partial G^+(x,y)}{\partial \nu_y}G_R(y,z,2)+\frac{\partial G_R}{\partial \nu_y}(y,z,2)G^+(x,y)\right\}d\sigma_y\nonumber\\
&=& \int_{\partial B_{\epsilon}(z)} \left\{-\frac{\partial G^+(x,y)}{\partial \nu_y}G_R(y,z,2)+\frac{\partial G_R}{\partial \nu_y}(y,z,2)G^+(x,y)\right\}d\sigma_y\nonumber\\
&\,&\,+\int_{\Gamma_R} \left\{\frac{\partial G_R}{\partial \nu_y}(y,z,2)G^+(x,y)\right\}d\sigma_y\nonumber
\end{eqnarray}

Next, we will estimate the above integrals on the boundary. For arbitrarily fixed $x \in B_R^+$, by mean value theorem, we have
\begin{eqnarray}\label{3-6-5}
G^+(x,y)&=&C_{n}\left(\ln{\frac{1}{|x-y|}}-\ln{\frac{1}{|\bar{x}-y|}}\right)\nonumber\\
&=& \frac{C_{n}}{2}\left\{\ln{\left(|\bar{x}-y|\right)^2}-\ln{\left(|x-y|\right)^2}\right\}\\
&=& C \frac{x_ny_n}{\xi^2} \sim \frac{y_n}{R^2},\,\,\quad \forall \, y\in \partial B^+_R, \quad\,\, \text{as}\, R\rightarrow +\infty,\nonumber
\end{eqnarray}
where $\xi$ is valued between $|x-y|^2$ and $|\bar{x}-y|^2$. By \eqref{3-5-2} and \eqref{3-6-5}, we obtain
\begin{equation}\label{3-6-6}
\left|\int_{\Gamma_R} \left\{\frac{\partial G_R}{\partial \nu_y}(y,z,2)G^+(x,y)\right\}d\sigma_y\right|\sim \frac{y_n^2}{R^{4}} \leq \frac{C}{R^{2}}\rightarrow 0,\,\,\quad \text{as} \,R\rightarrow+\infty.
\end{equation}

Since $\frac{\partial G_R}{\partial \nu_y}(y,z,2) \sim \frac{1}{\epsilon^{n-1}}$, $G_R(y,z,2) \sim \frac{1}{\epsilon^{n-2}}$ on $\partial B_{\epsilon}(z)$ and $G^+(x,y)$ is smooth in $B_{\epsilon}(z)$ for $\epsilon$ sufficiently small, we have
\begin{equation}\label{3-6-7}
\int_{\partial B_{\epsilon}(z)}\left\{-\frac{\partial G^+(x,y)}{\partial \nu_y}G_R(y,z,2)+\frac{\partial G_R}{\partial \nu_y}(y,z,2)G^+(x,y)\right\}d\sigma_y \rightarrow G^+(x,z),
\end{equation}
as $\epsilon \rightarrow 0+$. First let $\epsilon \rightarrow 0+$ and then let $ R\rightarrow+\infty$, combining \eqref{3-6-4}, \eqref{3-6-6} and \eqref{3-6-7}, we derive
\begin{equation}\label{3-6-8}
\text{LHS}\rightarrow G^+(x,z).
\end{equation}

As to the right-hand side, we have
\begin{equation}\label{3-6-9}
\text{RHS}=\int_{B_R^+\setminus B_{\epsilon}(z)} G(x,y,n-2)G_R(y,z,2)dy.
\end{equation}

Since $G_R(y,z,2)$ is smooth in $B_R^+\setminus B_{\epsilon}(z)$, $G(x,y,n-2) \sim \frac{1}{|x-y|^2}$ near $x=y$ and $n\geq4$, we derive from \eqref{3-6-8} that
\begin{equation}\label{3-16}
\int_{B_R^+\setminus B_{\epsilon}(z)} G(x,y,n-2)G_R(y,z,2)dy \leq C <+\infty.
\end{equation}
First let $\epsilon \rightarrow 0$ and then let $ R\rightarrow \infty$, by \eqref{3-4} and dominated convergence theorem, we deduce
\begin{equation}\label{3-6-10}
RHS=\int_{B_R^+\setminus B_{\epsilon}(z)} G(x,y,n-2)G_R(y,z,2)dy \rightarrow \int_{\mathbb{R}_+^n} G(x,y,n-2)G(y,z,2)dy
\end{equation}

Combining \eqref{3-6-8} and \eqref{3-6-10}, we arrive at
\begin{equation}\label{3-6-11}
G^+(x,z)= \int_{\mathbb{R}_+^n} G(x,y,n-2)G(y,z,2)dy
\end{equation}
This completes the proof of Lemma \ref{lem3-0}.
\end{proof}

Now we can prove $u$ is also a solution of the integral equation \eqref{IE} using Lemma \ref{lem3-0}. Multiplying both sides of \eqref{PDES} by $G_R(x,y,2)$ and integrating on $B_R^+$ by parts, we have
\begin{equation}\label{3-7}
\left\{{\begin{array}{l} {\int_{B_R^+}G_R(x,y,2)|y|^{a}u^{p}(y)dy=v_{\frac{n}{2}-1}(x)+\int_{\Gamma_R}v_{\frac{n}{2}-1}(y)\frac{\partial G_R(x,y,2)}{\partial \nu_y}d\sigma_y },  \\{}\\
	{\int_{B_R^+}G_R(x,y,2) v_{\frac{n}{2}-1}(y)dy=v_{\frac{n}{2}-2}(x)+\int_{\Gamma_R}v_{\frac{n}{2}-2}(y)\frac{\partial G_R(x,y,2)}{\partial \nu_y}d\sigma_y }, \\{}\\
	\cdots\cdots \\ {}\\
	{\int_{B_R^+}G_R(x,y,2)v_1(y)dy= u(x)+\int_{\Gamma_R}u(y)\frac{\partial G_R(x,y,2)}{\partial \nu_y}d\sigma_y }. \\ \end{array}}\right.
\end{equation}
By Theorem \ref{Thm0}, \eqref{3-5} and \eqref{3-7}, we derive
\begin{equation}\label{3-8}
\left\{{\begin{array}{l} {\int_{B_R^+}G_R(x,y,2)|y|^{a}u^{p}(y)dy\leq v_{\frac{n}{2}-1}(x)},  \\{}\\
	{\int_{B_R^+}G_R(x,y,2) v_{\frac{n}{2}-1}(y)dy\leq v_{\frac{n}{2}-2}(x)}, \\{}\\
	\cdots\cdots \\ {}\\
	{\int_{B_R^+}G_R(x,y,2)v_1(y)dy\leq u(x)}. \\ \end{array}}\right.
\end{equation}
Letting $R\rightarrow \infty$, and by \eqref{3-4}, we deduce
\begin{equation}\label{3-9}
\left\{{\begin{array}{l} {\int_{\mathbb{R}_+^n}G(x,y,2)|y|^{a}u^{p}(y)dy<\infty},  \\{}\\
	{\int_{\mathbb{R}_+^n}G(x,y,2) v_{\frac{n}{2}-1}(y)dy<\infty}, \\{}\\
	\cdots\cdots \\ {}\\
	{\int_{\mathbb{R}_+^n}G(x,y,2)v_1(y)dy<\infty}. \\ \end{array}}\right.
\end{equation}
By \eqref{3-9}, we conclude that there exists a sequence $R_j\rightarrow \infty$ such that
\begin{align}
&R_j\int_{\Gamma_{R_j}}G(x,y,2)v_i(y)d\sigma_y\rightarrow 0, \quad\quad as\,\, R_j\rightarrow \infty, \,\,\quad i=1,\ldots,\frac{n}{2}-1,\label{3-10-1}\\
&R_j\int_{\Gamma_{R_j}}G(x,y,2)|y|^{a}u(y)d\sigma_y\rightarrow 0, \quad\quad as\,\, R_j\rightarrow \infty.\label{3-10-2}
\end{align}
Then from \eqref{3-5-1}, it follows easily that
\begin{align}
&\frac{1}{R_j^{n-1}}\int_{\Gamma_{R_j}}v_i(y)y_n d\sigma_y\rightarrow 0, \quad\quad as\,\, R_j\rightarrow \infty, \,\,\quad i=1,\ldots,\frac{n}{2}-1,\label{3-11-1}\\
&\frac{1}{R_j^{n-1-a}}\int_{\Gamma_{R_j}}u^{p}(y)y_{n}d\sigma_y\rightarrow 0, \quad\quad as\,\, R_j\rightarrow \infty.\label{3-11-2}
\end{align}
As an immediate consequence of \eqref{3-11-1}, we have
\begin{equation}\label{3-12}
\frac{1}{R_j^{n+1}}\int_{\Gamma_{R_j}}v_i(y)y_n d\sigma_y\rightarrow 0, \quad\quad as\,\, R_j\rightarrow \infty, \,\,\quad i=1,\ldots,\frac{n}{2}-1.
\end{equation}
By \eqref{3-11-2}, H\"{o}lder inequality and the fact that $1+a+p>0$, we derive
\begin{eqnarray}\label{3-13}
\frac{1}{R_j^{n+1}}\int_{\Gamma_{R_j}}u(y)y_n d\sigma_y
&\leq& \frac{R_j^{\frac{n-1-a}{p}}}{R_j^{n+1}}{\left(\frac{1}{R_j^{n-1-a}}\int_{\Gamma_{R_j}}u^{p}(y)y_n d\sigma_y\right)}^{\frac{1}{p}} {\left(\int_{\Gamma_{R_j}}y_n d\sigma_y\right)}^{1-\frac{1}{p}}\nonumber\\
&\leq& \frac{R_j^{\frac{n-1-a}{p}}}{R_j^{n+1}}{\left(\frac{1}{R_j^{n-1-a}}\int_{\Gamma_{R_j}}u^{p}(y)y_n d\sigma_y\right)}^{\frac{1}{p}} R_j^{n(1-\frac{1}{p})}\\
&\leq& R_j^{-\frac{1+a+p}{p}}o(1) \rightarrow 0, \quad\,\,\, as\,\, R_j\rightarrow \infty.\nonumber
\end{eqnarray}
Substituting \eqref{3-5-2} into \eqref{3-7}, and by \eqref{3-4}, \eqref{3-12}, \eqref{3-13}, we arrive at
\begin{equation}\label{3-14}
\left\{{\begin{array}{l} {v_{\frac{n}{2}-1}(x)=\int_{\mathbb{R}_+^n}G(x,y,2)|y|^{a}u^{p}(y)dy,}  \\{}\\
	{v_{\frac{n}{2}-2}(x)=\int_{\mathbb{R}_+^n}G(x,y,2) v_{\frac{n}{2}-1}(y)dy, }\\{}\\
	\cdots\cdots \\ {}\\
	{u(x)=\int_{\mathbb{R}_+^n}G(x,y,2)v_1(y)dy.} \\ \end{array}}\right.
\end{equation}
Then by \eqref{3-6} and Lemma \ref{lem3-0}, it is easy to see that
\begin{eqnarray}\label{3-15}
u(x)&=&\int_{\mathbb{R}_+^n}G(x,y,2)v_1(y)dy \\
\nonumber &=&\int_{\mathbb{R}_+^n}G(x,y,2)\int_{\mathbb{R}_+^n}G(y,z,2)v_2(z)dzdy  \\
\nonumber &=&\int_{\mathbb{R}_+^n}G(x,y,4)v_2(y)dy\\
\nonumber &=&\cdots\\
\nonumber &=&\int_{\mathbb{R}_+^n} G^+(x,y)|y|^{a}u^{p}(y)dy.
\end{eqnarray}

\emph{(ii)} Next, we consider the case $n=2$.

For $n=2$, the Green's function $G^{+}_R(x,y)$ associated with $-\Delta$ for the half disk $B_R^+$ is given by
\begin{equation}\label{Greenball}
G^{+}_R(x,y):=\frac{1}{2\pi}\left(\ln\frac{1}{|x-y|}-\ln\frac{1}{\left|\frac{R^{2}x}{|x|^{2}}-y\right|}\right)
-\frac{1}{2\pi}\left(\ln\frac{1}{|\bar{x}-y|}-\ln\frac{1}{\left|\frac{R^{2}\bar{x}}{|x|^{2}}-y\right|}\right).
\end{equation}
Now we give some essential properties of the above Green's functions. First,
\begin{equation}\label{3-4p}
G^{+}_R(x,y)\rightarrow G^{+}(x,y)=\frac{1}{2\pi}\left(\ln\frac{1}{|x-y|}-\ln\frac{1}{|\bar{x}-y|}\right), \,\,\,\,\,\,\text{as}\,\,R\rightarrow+\infty.
\end{equation}
Second, let $\Gamma_R$ be the semi-circle part of $\partial B_R^+$, then for each fixed $x \in B_R^+$ and any $y \in \Gamma_R$,
\begin{equation}\label{3-5p}
\frac{\partial G^{+}_R(x,y)}{\partial\nu_y}=-\frac{1}{2\pi}R\left(1-\frac{|x|^2}{R^2}\right)\left(\frac{1}{|x-y|^{2}}-\frac{1}{|\bar{x}-y|^{2}} \right)<0,
\end{equation}
moreover, for each fixed $x \in B_R^+$, any $y \in \Gamma_R$ and $R$ sufficiently large,
\begin{gather}
G^{+}(x,y)=\frac{1}{2\pi}\left(\ln\frac{1}{|x-y|}-\ln\frac{1}{|\bar{x}-y|}\right) \sim \frac{y_2}{R^{2}}, \quad \text{as} \,\,\, R\rightarrow+\infty,\label{3-5-1p}\\
\left|\frac{\partial G^{+}_R(x,y)}{\partial\nu_y}\right|=\frac{1}{2\pi}R\left(1-\frac{|x|^2}{R^2}\right)\left(\frac{1}{|x-y|^{2}}-\frac{1}{|\bar{x}-y|^{2}} \right) \sim \frac{y_2}{R^{3}}.\label{3-5-2p}
\end{gather}

Multiplying both sides of PDE \eqref{NPDE} by $G^{+}_R(x,y)$ and integrating on $B_R^+$ by parts, we have
\begin{equation}\label{3-7p}
\int_{B_R^+}G^{+}_R(x,y)|y|^{a}u^{p}(y)dy=u(x)+\int_{\Gamma_R}u(y)\frac{\partial G^{+}_R(x,y)}{\partial \nu_y}d\sigma_y.
\end{equation}
By \eqref{3-5p} and \eqref{3-7p}, we derive
\begin{equation}\label{3-8p}
\int_{B_R^+}G^{+}_R(x,y)|y|^{a}u^{p}(y)dy\leq u(x).
\end{equation}
By letting $R\rightarrow \infty$ and \eqref{3-4p}, we deduce
\begin{equation}\label{3-9p}
\int_{\mathbb{R}_+^2}G^{+}(x,y)|y|^{a}u^{p}(y)dy<\infty.
\end{equation}
By \eqref{3-9p}, we conclude that there exists a sequence $R_j\rightarrow \infty$ such that
\begin{equation}\label{3-10-2p}
R_j\int_{\Gamma_{R_j}}G^{+}(x,y)|y|^{a}u(y)d\sigma_y\rightarrow 0, \quad\quad as\,\, R_j\rightarrow \infty,
\end{equation}
and hence, from \eqref{3-5-1p}, it follows easily that
\begin{equation}\label{3-11-2p}
\frac{1}{R_j^{1-a}}\int_{\Gamma_{R_j}}u^{p}(y)y_{2}d\sigma_y\rightarrow 0, \quad\quad as\,\, R_j\rightarrow \infty.
\end{equation}
By \eqref{3-11-2p}, H\"{o}lder inequality and the fact that $1+a+p>0$, we derive
\begin{eqnarray}\label{3-13p}
\frac{1}{R_j^{3}}\int_{\Gamma_{R_j}}u(y)y_2 d\sigma_y
&\leq& \frac{R_j^{\frac{1-a}{p}}}{R_j^{3}}{\left(\frac{1}{R_j^{1-a}}\int_{\Gamma_{R_j}}u^{p}(y)y_2 d\sigma_y\right)}^{\frac{1}{p}} {\left(\int_{\Gamma_{R_j}}y_2 d\sigma_y\right)}^{1-\frac{1}{p}}\nonumber\\
&\leq& \frac{R_j^{\frac{1-a}{p}}}{R_j^{3}}{\left(\frac{1}{R_j^{1-a}}\int_{\Gamma_{R_j}}u^{p}(y)y_2 d\sigma_y\right)}^{\frac{1}{p}} R_j^{2(1-\frac{1}{p})}\\
&\leq& R_j^{-\frac{1+a+p}{p}}o(1) \rightarrow 0, \quad\,\,\, as\,\, R_j\rightarrow \infty.\nonumber
\end{eqnarray}
Substituting \eqref{3-5-2p} into \eqref{3-7p}, and by \eqref{3-4p}, \eqref{3-13p}, we finally arrive at
\begin{equation}\label{3-14p}
u(x)=\int_{\mathbb{R}_+^2}G^{+}(x,y)|y|^{a}u^{p}(y)dy.
\end{equation}

This completes our proof of Theorem \ref{equivalence}.

\section{The proof of theorem \ref{thmIE}}

In this section, we will carry out the proof of Theorem \ref{thmIE}	by applying the method of scaling spheres in integral forms developed by Dai and Qin in \cite{DQ0}.

Suppose $u$ is a nonnegative continuous solution of IE \eqref{IE} but $u\not\equiv0$, we will derive a contradiction via the method of scaling spheres in integral forms.

In order to apply the method of scaling spheres, we first give some definitions. One can easily see that $u\not\equiv0$ implies $u>0$ in $\mathbb{R}^{n}_{+}$. Let $\lambda>0$ be an arbitrary positive real number and let the scaling half sphere be
\begin{equation}\label{4-1}
S^{+}_{\lambda}=\{x\in\overline{\mathbb{R}^{n}_+}:\, |x|=\lambda\}.
\end{equation}
We denote the reflection of $x$ about the sphere $\{x\in\mathbb{R}^{n}\,|\,|x|=\lambda\}$ by $x^{\lambda}:=\frac{\lambda^{2}x}{|x|^{2}}$ and let
\begin{equation}\label{2-5+}
  B^{+}_{\lambda}(0):=B_{\lambda}(0)\cap\mathbb{R}^{n}_{+}, \quad\quad \widetilde{B^{+}_{\lambda}}(0):=\{x\in\mathbb{R}^{n}_{+}: \, x^{\lambda}\in B^{+}_{\lambda}(0)\}.
\end{equation}
Define the Kelvin transform of $u$ centered at $0$ by
\begin{equation}\label{Kelvin}
u_{\lambda}(x)=u\left(\frac{\lambda^{2}x}{|x|^{2}}\right)
\end{equation}
for arbitrary $x\in\overline{\mathbb{R}^{n}_{+}}\setminus\{0\}$. It's obvious that the Kelvin transform $u_{\lambda}$ may have singularity at $0$ and $\lim_{|x|\rightarrow\infty}u_{\lambda}(x)=u(0)=0$. By \eqref{Kelvin}, one can infer from the regularity assumptions on $u$ that $u_{\lambda}\in C(\overline{\mathbb{R}^{n}_{+}}\setminus\{0\})$.

By direct calculations, one can verify that
\begin{gather}
G^+(x,y)>G^+(x,y^\lambda),\,\,\,\,\,\,\,\,\forall \,\, x,y\in B_{\lambda}^+(0),\label{4-2}\\
\left|x^\lambda-y\right|=\frac{|y|}{|x|}\left|x-y^\lambda\right|, \,\,\,|x|=|\bar{x}|,\,\,\,\overline{x^\lambda}={\bar{x}}^\lambda.\label{4-3}
\end{gather}

We can deduce from \eqref{IE}, \eqref{Kelvin} and \eqref{4-3} that (for the invariance properties of fractional or higher order Laplacians under the Kelvin type transforms, please refer to \cite{CGS,CL,CLO,Lin,WX})
\begin{eqnarray}\label{4-4}
u_\lambda(x)&=&u\left(\frac{\lambda^{2}x}{|x|^{2}}\right)=\int_{\mathbb{R}^n_+} G^+(x^{\lambda},y)|y|^au^p(y) dy\\
&=& C_{n}\int_{\mathbb{R}^n_+} \left(\ln{\frac{1}{|x^\lambda-y|}}-\ln{\frac{1}{|\overline{x^\lambda}-y|}}\right)|y|^au^p(y) dy\nonumber\\
&=& C_{n}\int_{\mathbb{R}^n_+} \left(\ln{\frac{1}{|x^\lambda-y|}}-\ln{\frac{1}{|\overline{x}^\lambda-y|}}\right)|y|^au^p(y) dy\nonumber\\
&=& C_{n}\int_{\mathbb{R}^n_+} \left(\ln{\frac{|x|}{|y||x-y^\lambda|}}-\ln{\frac{|\bar{x}|}{|y||\overline{x}-y^\lambda|}}\right)|y|^au^p(y) dy\nonumber\\
&=& C_{n}\int_{\mathbb{R}^n_+} \left(\ln{\frac{1}{|x-y^\lambda|}}-\ln{\frac{1}{|\overline{x}-y^\lambda|}}\right)|y|^au^p(y) dy\nonumber\\
&=& \int_{\mathbb{R}^n_+} G^+(x,y^\lambda)|y|^au^p(y) dy.\nonumber
\end{eqnarray}

Let $\omega^{\lambda}(x):=u_{\lambda}(x)-u(x)$ for any $x\in\overline{B^{+}_{\lambda}(0)}\setminus\{0\}$. Since $u$ satisfies \eqref{IE}, by changing variables, we have
\begin{align}\label{4-5}
   u(x)&=\int_{\mathbb{R}^{n}_{+}} G^{+}(x,y)|y|^{a} u^{p}(y)dy\\
   &=\int_{B^{+}_\lambda(0)} G^{+}(x,y)|y|^{a} u^{p}(y)dy+\int_{\mathbb{R}^{n}_{+}\setminus B^{+}_\lambda(0)} G^{+}(x,y)|y|^{a} u^{p}(y)dy\nonumber\\
   &=\int_{B^{+}_\lambda(0)} G^{+}(x,y)|y|^{a} u^{p}(y)dy+\int_{B^{+}_\lambda(0)} {\left(\frac{\lambda}{|y|}\right)}^{2(n+a)}G^{+}(x,y^\lambda)|y|^{a} u_\lambda^{p}(y)dy.\nonumber
\end{align}
Similarly, by \eqref{4-4}, we obtain
\begin{align}\label{4-6}
u_\lambda(x)&=\int_{\mathbb{R}^{n}_{+}} G^{+}(x,y^\lambda)|y|^{a}u^{p}(y)dy\\
&=\int_{B^{+}_\lambda(0)} G^{+}(x,y^\lambda)|y|^{a}u^{p}(y)dy +\int_{B^{+}_\lambda(0)} {\left(\frac{\lambda}{|y|}\right)}^{2(n+a)}G^{+}(x,y)|y|^{a} u_{\lambda}^{p}(y)dy.\nonumber
\end{align}
Then, by \eqref{4-5} and \eqref{4-6}, we arrive at
\begin{align}\label{4-7}
\omega^{\lambda}(x)&=u_{\lambda}(x)-u(x)\\
&=\int_{B^{+}_\lambda(0)}\left(G^{+}(x,y)-G^{+}(x,y^\lambda)\right)|y|^{a}\left( {\left(\frac{\lambda}{|y|}\right)}^{2(n+a)}u_\lambda^{p}(y)-u^{p}(y)\right)dy \nonumber
\end{align}
for every $x\in B^{+}_{\lambda}(0)$.

Now we can carry out the process of scaling spheres in two steps.

\emph{Step 1. Start dilating the sphere from near $\lambda=0$.}
We will first show that, for $\lambda>0$ sufficiently small,
\begin{equation}\label{4-8}
\omega^\lambda(x)\geq 0,\,\,\,\quad \forall \,\, x\in B^{+}_{\lambda}(0).
\end{equation}
Define
\begin{equation}\label{4-9}
(B^{+}_\lambda)^-:=\{x\in B^{+}_{\lambda}(0)\,|\,\omega_\lambda(x)<0\}.
\end{equation}

Through elementary calculations, one can obtain that for any $x,\,y \in B_{\lambda}^+(0)$, $x\neq y$,
\begin{eqnarray}\label{GP}
G^+(x,y)&=&C_{n}\left(\ln{\frac{1}{|x-y|}}-\ln{\frac{1}{|\bar{x}-y|}}\right)\\
&=& \frac{C_{n}}{2} \ln{\frac{|\bar{x}-y|^2}{|x-y|^2}}\nonumber\\
&=& C \ln{\left(1+\frac{4x_ny_n}{|x-y|^2}\right)} \nonumber\\
&\leq& C \ln{\left(1+\frac{4\lambda^2}{|x-y|^2}\right)}. \nonumber
\end{eqnarray}
It is well known that
\begin{equation}\label{wk}
\ln{(1+t)}=o(t^\varepsilon), \,\,\quad \text{as}\,\, t\rightarrow +\infty,
\end{equation}
where $\varepsilon$ is an arbitrary positive real number. This implies, for any given $\varepsilon>0$, there exists a $\delta(\varepsilon)>0$ such that
\begin{equation}\label{ln}
\ln{(1+t)}\leq t^\varepsilon, \,\,\qquad \forall \, t>\frac{4}{{\delta(\varepsilon)}^2}.
\end{equation}

Therefore, by \eqref{GP}, \eqref{ln} and straightforward calculations, we have the following lemma that states some basic estimates for Green's function $G^+(x,y)$.
\begin{lem}\label{Gr}
Assume $G^{+}(x,y)$ be the Green's functions in integral equation \eqref{IE}. Then we have
\begin{gather}
G^+(x,y)\leq C\lambda^{2\varepsilon} \frac{1}{|x-y|^{2\varepsilon}}, \,\,\qquad \forall\,x,\,y \in B_{\lambda}^+(0), \,\,|x-y|<\lambda \delta(\varepsilon);\label{GP-1}\\
G^+(x,y)\leq C \ln\left(1+\frac{4}{{\delta(\varepsilon)}^2}\right), \,\,\qquad \forall\,x,\,y \in B_{\lambda}^+(0), \,\, |x-y|\geq\lambda \delta(\varepsilon);\label{GP-2}\\
G^+(x,y)\leq C' \frac{4x_ny_n}{|x-y|^{2}}, \,\,\qquad \forall \, x,y\in\mathbb{R}^{n}_{+},\,\, x\neq y;\label{GP-3}\\
G^+(x,y)\geq C'' \frac{4x_ny_n}{|x-y|^{2}}, \,\,\,\qquad \forall \, x,y\in\mathbb{R}^{n}_{+},\,\,\,\, \frac{|x|}{|y|}\leq\frac{1}{100} \,\,\, \text{or} \,\,\, \frac{|y|}{|x|}\leq\frac{1}{100}.\label{GP-4}
\end{gather}
\end{lem}

By the assumption $a>-n$, \eqref{4-2}, \eqref{4-7} \eqref{GP-1} and \eqref{GP-2}, we have, for any $x\in (B^{+}_\lambda)^-$,
\begin{eqnarray}\label{4-10}
\omega^{\lambda}(x)&=&\int_{B^{+}_\lambda(0)} \left(G^{+}(x,y)-G^{+}(x,y^\lambda)\right)|y|^{a}\left( {\left(\frac{\lambda}{|y|}\right)}^{2(n+a)}u_\lambda^{p}(y)-u^{p}(y)\right)dy\nonumber\\
&>& \int_{B^{+}_\lambda(0)} \left(G^{+}(x,y)-G^{+}(x,y^\lambda)\right)|y|^{a}\left( u_\lambda^{p}(y)-u^{p}(y)\right)dy\nonumber\\
&\geq& \int_{(B^{+}_\lambda)^-} \left(G^{+}(x,y)-G^{+}(x,y^\lambda)\right)|y|^{a}\left(u_\lambda^{p}(y)-u^{p}(y)\right)dy\\
&\geq& \int_{(B^{+}_\lambda)^-} G^{+}(x,y)|y|^{a}\left( u_\lambda^{p}(y)-u^{p}(y)\right)dy\nonumber\\
&\geq& C\lambda^{2\varepsilon}\int_{(B^{+}_\lambda)^-\,\cap\,B_{\lambda \delta(\varepsilon)}(x)}\frac{1}{|x-y|^{2\varepsilon}} |y|^{a}u^{p-1}(y)\omega^{\lambda}(y)dy\nonumber\\
&\,& +C(\delta(\varepsilon))\int_{(B^{+}_\lambda)^-\setminus B_{\lambda \delta(\varepsilon)}(x)} |y|^{a}u^{p-1}(y)\omega^{\lambda}(y)dy.\nonumber
\end{eqnarray}

By \eqref{4-10}, Hardy-Littlewood-Sobolev inequality and H\"{o}lder inequality, for arbitrary $\frac{n}{2\varepsilon}<q<\infty$, we obtain
\begin{eqnarray}\label{4-11}
{\| \omega^{\lambda}\|}_{L^q((B^{+}_{\lambda})^-)}
&\leq& C\lambda^{2\varepsilon}{\left\| {\int_{(B^{+}_{\lambda})^-\,\cap\,B_{\lambda \delta(\varepsilon)}(x)}\frac{1}{|x-y|^{2\varepsilon}} |y|^{a}u^{p-1}(y)\omega^{\lambda}(y)dy}\right\|}_{L^q((B^{+}_{\lambda})^-)}\nonumber\\
&\,& +C(\delta(\varepsilon))|(B^{+}_{\lambda})^-|^{\frac{1}{q}}\int_{(B^{+}_{\lambda})^-} |y|^{a}u^{p-1}(y)|\omega^{\lambda}(y)|dy\nonumber\\
&\leq& C\lambda^{2\varepsilon}{\left\| {\int_{(B^{+}_{\lambda})^-}\frac{1}{|x-y|^{2\varepsilon}} |y|^{a}u^{p-1}(y)\omega^{\lambda}(y)dy}\right\|}_{L^q((B^{+}_{\lambda})^-)}\nonumber\\
&\,& +C(\delta(\varepsilon))|(B^{+}_{\lambda})^-|^{\frac{1}{q}}\int_{(B^{+}_{\lambda})^-} |y|^{a}u^{p-1}(y)|\omega^{\lambda}(y)|dy\\
&\leq& C\lambda^{2\varepsilon}{\left\| |x|^{a}u^{p-1}\omega^{\lambda}\right\|}_{L^{\frac{nq}{n+(n-2\varepsilon)q}}((B^{+}_{\lambda})^-)}\nonumber\\
&\,& +C(\delta(\varepsilon))|(B^{+}_{\lambda})^-|^{\frac{1}{q}}\int_{(B^{+}_{\lambda})^-} |y|^{a}u^{p-1}(y)|\omega^{\lambda}(y)|dy\nonumber\\
&\leq& C\lambda^{2\varepsilon}{\left\| |x|^{a}u^{p-1}\right\|}_{L^{\frac{n}{n-2\varepsilon}}((B^{+}_{\lambda})^-)}{\left\| \omega^{\lambda}\right\|}_{L^{q}((B^{+}_{\lambda})^-)}\nonumber\\
&\,& +C(\delta(\varepsilon))|(B^{+}_{\lambda})^-|^{\frac{1}{q}}{\left\| |x|^{a}u^{p-1}\right\|}_{L^{\frac{q}{q-1}}((B^{+}_{\lambda})^-)}{\left\| \omega^{\lambda}\right\|}_{L^{q}((B^{+}_{\lambda})^-)}.\nonumber
\end{eqnarray}
Since $a>-n$, we first choose $\varepsilon>0$ sufficiently small such that $\frac{na}{n-2\varepsilon}>-n$, then choose $q>\frac{n}{2\varepsilon}$ sufficiently large such that $\frac{qa}{q-1}>-n$. Then from the assumption that $u$ is continuous in $\overline{\mathbb{R}^{n}_{+}}$, it is obvious that $|x|^{a}u^{p-1} \in L^{\frac{n}{n-2\varepsilon}}_{loc}(\overline{\mathbb{R}^{n}_{+}})\cap L^{\frac{q}{q-1}}_{loc}(\overline{\mathbb{R}^{n}_{+}})$. Therefore, there exists $\delta_0>0$ sufficiently small, such that
\begin{equation}\label{7-2}
  C\lambda^{2\varepsilon}{\||x|^{a}u^{p-1}\|}_{L^{\frac{n}{n-2\varepsilon}}((B^{+}_{\lambda})^-)}+C(\delta(\varepsilon))|(B^{+}_{\lambda})^-|^{\frac{1}{q}}{\left\| |x|^{a}u^{p-1}\right\|}_{L^{\frac{q}{q-1}}((B^{+}_{\lambda})^-)}<\frac{1}{2}
\end{equation}
for any $0<\lambda<\delta_0$. Thus, by \eqref{4-11}, we must have
\begin{equation}\label{7-1}
  {\| \omega^{\lambda}\|}_{L^q((B^{+}_{\lambda})^-)}=0
\end{equation}
for any $0<\lambda<\delta_0$. From the continuity of $\omega^{\lambda}$ in $\overline{\mathbb{R}^{n}_{+}}\setminus\{0\}$ and the definition of $(B^{+}_{\lambda})^-$, we immediately get that $(B^{+}_{\lambda})^-=\emptyset$ and hence \eqref{4-8} holds true for any $0<\lambda<\delta_0$. This completes Step 1.

\emph{Step 2. Dilate the half sphere $S^{+}_{\lambda}$ outward until $\lambda=+\infty$ to derive lower bound estimates on $u$ in a cone.} Step 1 provides us a starting point to dilate the half sphere $S^{+}_{\lambda}$ from near $\lambda=0$. Now we dilate the half sphere $S^{+}_{\lambda}$ outward as long as \eqref{4-8} holds. Let
\begin{equation}\label{4-12}
  \lambda_{0}:=\sup\{\lambda>0\,|\, \omega^{\mu}\geq0 \,\, \text{in} \,\, B^{+}_{\mu}(0), \,\, \forall \, 0<\mu\leq\lambda\}\in(0,+\infty],
\end{equation}
and hence, one has
\begin{equation}\label{4-13}
  \omega^{\lambda_{0}}(x)\geq0, \quad\quad \forall \,\, x\in B^{+}_{\lambda_{0}}(0).
\end{equation}

In what follows, we will prove $\lambda_{0}=+\infty$ by driving a contradiction under the assumption that $\lambda_0<+\infty$.

In fact, suppose $\lambda_0<+\infty$, we must have
\begin{equation}\label{4-14}
  \omega^{\lambda_{0}}(x)\equiv 0, \quad\quad \forall \,\, x\in B^{+}_{\lambda_{0}}(0).
\end{equation}
Suppose on the contrary that \eqref{4-14} does not hold, that is, there exists a $x_0 \in B^{+}_{\lambda_{0}}(0)$ such that $\omega^{\lambda_{0}}(x_0)>0$. Then, by \eqref{4-10} and \eqref{4-13}, we have
\begin{equation}\label{4-15}
  \omega^{\lambda_{0}}(x)>0, \quad\quad \forall \,\, x\in B^{+}_{\lambda_{0}}(0).
\end{equation}

Choose $\delta_1>0$ sufficiently small, which will be determined later. Define the narrow region
\begin{equation}\label{4-16}
A_{\delta_1}:=\{x\in B^{+}_{\lambda_{0}}(0)\,|\,dist(x,\partial B^{+}_{\lambda_{0}}(0))<\delta_1\}.
\end{equation}
Note that $A_{\delta_1}=\{x\in B^{+}_{\lambda_{0}}(0)\,|\,x_n<\delta_1\,\,\,\text{or}\,\,\,\lambda_0-\delta_1<|x|<\lambda_0\}$.

Since that $\omega^{\lambda_{0}}$ is continuous in $\overline{\mathbb{R}^{n}_{+}}\setminus\{0\}$ and $A_{\delta_1}^c:=B^{+}_{\lambda_0}(0)\setminus A_{\delta_1} $ is a compact subset, there exists a positive constant $C_0$ such that
\begin{equation}\label{4-17}
  \omega^{\lambda_{0}}(x)>C_0, \quad\quad \forall \,\, x\in A_{\delta_1}^c.
\end{equation}
By continuity, we can choose $\delta_2>0$ sufficiently small, such that, for any $\lambda\in [\lambda_0,\,\lambda_0+\delta_2]$,
\begin{equation}\label{4-18}
  \omega^{\lambda}(x)>\frac{C_0}{2}, \quad\quad \forall \,\, x\in A_{\delta_1}^c.
\end{equation}
Hence we must have
\begin{equation}\label{7-3}
  (B^{+}_{\lambda})^-\subset B^{+}_{\lambda}(0)\setminus A^{c}_{\delta_1}:=\left(B^{+}_{\lambda}(0)\setminus B^{+}_{\lambda_{0}}(0)\right)\cup A_{\delta_{1}}
\end{equation}
for any $\lambda\in [\lambda_0,\,\lambda_0+\delta_2]$. By \eqref{4-11} and local integrability of $|x|^{a}u^{p-1}$ in $\overline{\mathbb{R}^{n}_{+}}$, we can choose $\delta_1$ and $\delta_2$ sufficiently small such that, for any $\lambda\in [\lambda_0,\,\lambda_0+\delta_2]$,
\begin{equation}\label{7-4}
  C\lambda^{2\varepsilon}{\||x|^{a}u^{p-1}\|}_{L^{\frac{n}{n-2\varepsilon}}((B^{+}_{\lambda})^-)}+C(\delta(\varepsilon))|(B^{+}_{\lambda})^-|^{\frac{1}{q}}{\left\| |x|^{a}u^{p-1}\right\|}_{L^{\frac{q}{q-1}}((B^{+}_{\lambda})^-)}<\frac{1}{2}.
\end{equation}
Then, by the same argument as in step 1, we obtain that for any $\lambda\in[\lambda_0,\,\lambda_0+\delta_2]$,
\begin{equation}\label{4-19}
  \omega^{\lambda}(x)\geq 0, \quad\quad \forall \,\, x\in B^{+}_{\lambda}(0).
\end{equation}
This contradicts the definition of $\lambda_0$. Hence, \eqref{4-14} must hold true.

However, by \eqref{4-2}, \eqref{4-7}, \eqref{4-14} and the fact that $n+a>0$, we obtain
\begin{eqnarray}\label{4-20}
0&=&\omega^{\lambda_{0}}(x)=u_{\lambda_0}(x)-u(x)\\
&=&\int_{B^{+}_{\lambda_0}(0)} \left(G^{+}(x,y)-G^{+}(x,y^{\lambda_0})\right)|y|^{a}\left( {\left(\frac{\lambda_0}{|y|}\right)}^{2(n+a)}-1\right)u^{p}(y)dy>0\nonumber
\end{eqnarray}
for any $x\in\, B^{+}_{\lambda_{0}}(0)$, which is absurd. Thus we must have $\lambda_0=+\infty$, that is,
\begin{equation}\label{4-21}
u(x)\geq u\left(\frac{\lambda^{2}x}{|x|^{2}}\right), \quad\quad \forall \,\, x\in\mathbb{R}^{n}_{+},\,\,\,|x|\geq\lambda, \quad \forall \,\, 0<\lambda<+\infty.
\end{equation}
Therefore, we obtain that $u$ is radially nondecreasing. For arbitrary $|x|\geq1$, $x\in\mathbb{R}^{n}_{+}$, let $\lambda:=\sqrt{|x|}$, then \eqref{4-21} yields that
\begin{equation}\label{4-22}
u(x)\geq u\left(\frac{x}{|x|}\right),
\end{equation}
and hence, we arrive at the following lower bound estimate:
\begin{equation}\label{4-23}
u(x)\geq\min_{|x|=1,\,x_{n}\geq\frac{1}{\sqrt{n}}}u(x):=C_{0}>0, \quad\quad \forall \,\, |x|\geq1, \,\,\, x_{n}\geq\frac{|x|}{\sqrt{n}}.
\end{equation}

The lower bound estimate \eqref{4-23} can be improved remarkably using the ``Bootstrap" iteration technique and the integral equation \eqref{IE}.

In fact, let $\mu_{0}:=0$, we infer from the integral equation \eqref{IE}, \eqref{GP-4} and \eqref{4-23} that, for any $|x|\geq1$ and $x_{n}\geq\frac{|x|}{\sqrt{n}}$,
\begin{eqnarray}\label{8-1}
  u(x)&\geq&\int_{0<|y|\leq\frac{1}{100}|x|,\,y_{n}\geq\frac{|y|}{\sqrt{n}}}G^{+}(x,y)|y|^{a}C^{p}_{0}|y|^{p\mu_{0}}dy \\
  \nonumber &\geq&C\int_{0<|y|\leq\frac{1}{100}|x|,\,y_{n}\geq\frac{|y|}{\sqrt{n}}}
  \frac{x_{n}y_{n}}{|x-y|^{2}}|y|^{p\mu_{0}+a}dy \\
  \nonumber &\geq&\frac{C}{|x|}\int^{\frac{|x|}{100}}_{0}r^{n+a+p\mu_{0}}dr \\
  \nonumber &\geq& C_{1}|x|^{p\mu_{0}+(n+a)},
\end{eqnarray}
where we have used the fact $(n+a)+p\mu_{0}>0$ since $a>-n$ and $1\leq p<+\infty$. Let $\mu_{1}:=p\mu_{0}+(n+a)$. Due to $1\leq p<+\infty$ and $a>-n$, our important observation is
\begin{equation}\label{8-2}
  \mu_{1}:=p\mu_{0}+(n+a)>\mu_{0}.
\end{equation}
Thus we have obtained a better lower bound estimate than \eqref{4-23} after one iteration, that is,
\begin{equation}\label{8-3}
  u(x)\geq C_{1}|x|^{\mu_{1}}, \quad\quad \forall \,\, |x|\geq1, \,\,\, x_{n}\geq\frac{|x|}{\sqrt{n}}.
\end{equation}

For $k=0,1,2,\cdots$, define
\begin{equation}\label{8-4}
  \mu_{k+1}:=p\mu_{k}+(n+a).
\end{equation}
Since $a>-n$ and $1\leq p<+\infty$, it is easy to see that the sequence $\{\mu_{k}\}$ is monotone increasing with respect to $k$ and $(n+a)+p\mu_{k}>0$ for any $k=0,1,2,\cdots$. Continuing the above iteration process involving the integral equation \eqref{IE}, we have the following lower bound estimates for every $k=0,1,2,\cdots$,
\begin{equation}\label{8-5}
  u(x)\geq C_{k}|x|^{\mu_{k}}, \quad\quad \forall \,\, |x|\geq1, \,\,\, x_{n}\geq\frac{|x|}{\sqrt{n}}.
\end{equation}
From \eqref{8-5} and the obvious property that
\begin{equation}\label{8-6}
   \mu_{k}\rightarrow+\infty, \quad\, \text{as} \,\, k\rightarrow+\infty,
\end{equation}
we can conclude the following lower bound estimates for positive solution $u$.
\begin{thm}\label{lower}
Assume $n\geq1$, $a>-n$ and $1\leq p<+\infty$. Suppose $u\in C(\overline{\mathbb{R}^{n}_{+}})$ is a positive solution to IE \eqref{IE}, then it satisfies the following lower bound estimates: for all $x\in\mathbb{R}^{n}_{+}$ satisfying $|x|\geq1$ and $x_{n}\geq\frac{|x|}{\sqrt{n}}$,
\begin{equation}\label{lb}
  u(x)\geq C_{\kappa}|x|^{\kappa}, \quad\quad \forall \, \kappa<+\infty.
\end{equation}
\end{thm}

The lower bound estimates in Theorem \ref{lower} obviously contradicts the integral equation \eqref{IE}. In fact, since $a>-n$, by \eqref{GP-4} and \eqref{4-23}, we have
\begin{eqnarray}
+\infty>u\left(\frac{e_n}{100}\right) &=& \int_{\mathbb{R}^{n}_{+}} G\left(\frac{e_n}{100},y\right)|y|^{a} u^{p}(y)dy\\
&\geq& \int_{ |y|\geq1, \, y_{n}\geq\frac{|y|}{\sqrt{n}}} G\left(\frac{e_n}{100},y\right)|y|^{a} u^{p}(y)dy \nonumber\\
&\geq& C\int_{ |y|\geq1, \, y_{n}\geq\frac{|y|}{\sqrt{n}}} \frac{y_n}{|y|^2}|y|^{a+1}dy \nonumber\\
&\geq& C\int_{ |y|\geq1, \, y_{n}\geq\frac{|y|}{\sqrt{n}}} |y|^{a} dy \nonumber\\
&=& +\infty\nonumber,
\end{eqnarray}
where the unit vector $e_{n}:=(0,\cdots,0,1)\in\mathbb{R}^{n}_{+}$. This is a contradiction! Thus we must have $u\equiv0$ in $\overline{\mathbb{R}^{n}_{+}}$.

This concludes our proof of Theorem \ref{thmIE}.

\end{document}